\documentclass[reqno,a4paper]{amsart}
\usepackage{amssymb}
\usepackage{amsmath}
\newcommand{\angles}[1]{\langle #1 \rangle}
\newcommand{\half}{\frac{1}{2}}

\begin{document} 
\newtheorem{prop}{Proposition}[section]
\newtheorem{Def}{Definition}[section] \newtheorem{theorem}{Theorem}[section]
\newtheorem{lemma}{Lemma}[section] \newtheorem{Cor}{Corollary}[section]

\title[Maxwell-Klein-Gordon in Lorenz gauge]{\bf Low regularity local well-posedness for the (N+1)-dimensional Maxwell-Klein-Gordon equations in Lorenz gauge}
\author[Hartmut Pecher]{
{\bf Hartmut Pecher}\\
Fachbereich Mathematik und Naturwissenschaften\\
Bergische Universit\"at Wuppertal\\
Gau{\ss}str.  20\\
42097 Wuppertal\\
Germany\\
e-mail {\tt pecher@math.uni-wuppertal.de}}
\date{}

\begin{abstract}
The Cauchy problem for the Maxwell-Klein-Gordon equations in Lorenz gauge in $n$ space dimensions ($n \ge 2$) is locally well-posed for low regularity data, in two and three space dimensions even for data without finite energy. The result relies on the null structure for the main bilinear terms which was shown to be not only present in Coulomb gauge but also in Lorenz gauge by Selberg and Tesfahun, who proved global well-posedness for finite energy data in three space dimensions. This null structure is combined with product estimates for wave-Sobolev spaces given systematically by d'Ancona, Foschi and Selberg.
\end{abstract}
\maketitle
\renewcommand{\thefootnote}{\fnsymbol{footnote}}
\footnotetext{\hspace{-1.5em}{\it 2000 Mathematics Subject Classification:} 
35Q61, 35L70 }
\normalsize 
\setcounter{section}{0}
\section{Introduction and main results}
\noindent Consider the Maxwell-Klein-Gordon system 
\begin{align}
\label{1}
\partial^{\nu} F_{\mu \nu} & =  j_{\mu} \\
\label{2}
D^{(A)}_{\mu} D^{(A)\mu} \phi & = m^2 \phi \, ,
\end{align}
where $m>0$ is a constant and
\begin{align}
\label{3}
F_{\mu \nu} & := \partial_{\mu} A_{\nu} - \partial_{\nu} A_{\mu} \\
\label{4}
D^{(A)}_{\mu} \phi & := \partial_{\mu} - iA_{\mu} \phi \\
\label{5}
j_{\mu} & := Im(\phi \overline{D^{(A)}_{\mu} \phi}) = Im(\phi \overline{\partial_{\mu} \phi}) + |\phi|^2 A_{\mu} \, .
\end{align}
Here $F_{\mu \nu} : {\mathbb R}^{n+1} \to {\mathbb R}$ denotes the electromagnetic field, $\phi : {\mathbb R}^{n+1} \to {\mathbb C}$ a scalar field and $A_{\nu} : {\mathbb R}^{n+1} \to {\mathbb R}$ the potential. We use the notation $\partial_{\mu} = \frac{\partial}{\partial x_{\mu}}$, where we write $(x^0,x^1,...,x^n) = (t,x^1,...,x^n)$ and also $\partial_0 = \partial_t$ and $\nabla = (\partial_1,...,\partial_n)$. Roman indices run over $1,...,n$ and greek indices over $0,...,n$ and repeated upper/lower indices are summed. Indices are raised and lowered using the Minkowski metric $diag(-1,1,...,1)$.

The Maxwell-Klein-Gordon system describes the motion of a spin 0 particle with mass $m$ self-interacting with an electromagnetic field.

We are interested in the Cauchy problem with data
$\phi(x,0) = \phi_0(x)$ , $\partial_t \phi(x,0) $ $= \phi_1(x)$ , $F_{\mu \nu}(x,0) =  F^0_{\mu \nu}(x)$ , $A_{\nu}(x,0) = a_{0 \nu}(x)$ , $\partial_t A_{\nu}(x,0) = \dot{a}_{0 \nu}(x)$. The potential $A$ is not uniquely determined but one has gauge freedom. The Maxwell-Klein-Gordon equation is namely invariant under the gauge transformation
$\phi \to \phi' = e^{i\chi}\phi$ , $A_{\mu} \to A'_{\mu} = A_{\mu} + \partial_{\mu} \chi$
for any $\chi: {\mathbb R}^{n+1} \to {\mathbb R}$.

Most of the results obtained so far were given in Coulomb gauge $\partial^j A_j = 0$. Klainerman and Machedon \cite{KM} showed global well-posedness in energy space and above, i.e. for data $\phi_0 \in H^s$ , $\phi_1 \in H^{s-1}$ , $a_{0 \nu} \in H^s$ , $\dot{a}_{0 \nu} \in H^{s-1}$ with $s \ge 1$ in $n=3$ dimensions improving earlier results of Eardley and Moncrief \cite{EM} for smooth data. They used that the nonlinearities fulfill a null condition in the case of the Coulomb gauge. This global well-posedness result was improved by Keel, Roy and Tao \cite{KRT}, who had only to assume $s > \frac{\sqrt 3}{2}$. Local well-posedness for low regularity data was shown by Cuccagna \cite{C} for $s > 3/4$ and finally almost down to the critical regularity with respect to scaling by Machedon and Sterbenz \cite{MS} for $s > 1/2$, all these results for three space dimensions and in Coulomb gauge. 

In two space dimensions in Coulomb gauge Czubak and Pikula \cite{CP} proved local well-posedness provided that $\phi_0 \in H^s$ , $\phi_1 \in H^{s-1}$ , $a_{0 \nu} \in H^r$ , $\dot{a}_{0 \nu} \in H^{r-1}$,  where $1 \ge s=r > \frac{1}{2}$ or $s=\frac{5}{8}+\epsilon$ , $r=\frac{1}{4}+\epsilon$. 

In four space dimensions Selberg \cite{S} showed local well-posedness in Coulomb gauge for $s>1$. Recently Krieger, Sterbenz and Tataru \cite{KST} showed global well-posedness for data with small energy data ($s=1$) for $n=4$, which is the critical space. For space dimension $n \ge 6$ and small critical Sobolev norm for the data local well-posedness was shown by Rodnianski and Tao \cite{RT}. In general the problem seems to be easier in higher dimensions. In temporal gauge local well-posedness was shown for $n=3$ and $ s > 3/4$ for the more general Yang-Mills equations by Tao \cite{T}.

We are interested to consider the Maxwell-Klein-Gordon equations in Lorenz gauge $\partial^{\mu} A_{\mu} = 0$ which was considered much less in the literature because the nonlinear term $Im(\phi \overline{\partial_{\mu} \phi})$ has no null structure. There is a result by Moncrief \cite{M} in two space dimensions for smooth data, i.e. $s \ge 2$. In three space dimensions the most important progress was made by Selberg and Tesfahun \cite{ST} who were able to circumvent the problem of the missing null condition in the equations for $A_{\mu}$ by showing that the decisive nonlinearities in the equations for $\phi$ as well as $F_{\mu \nu}$ fulfill such a null condition which allows to show that global well-posedness holds for finite energy data, i.e. $\phi_0 \in H^1$, $\phi_1 \in L^2$ , $F^0_{\mu \nu} \in L^2$ , $a_{0\nu} \in \dot{H}^1$ , $\dot{a}_{0 \nu} \in L^2$, and three space dimensions, where $\phi \in C^0({\mathbb R},H^1) \cap C^1({\mathbb R},L^2)$ and $F_{\mu \nu} \in C^0({\mathbb R},L^2)$. The potential possibly loses some regularity compared to the data but as remarked also by the authors this is not the main point because one is primarily interested in the regularity of $\phi$ and $F_{\mu \nu}$. Persistence of higher regularity for the solution also holds.

A null structure in Lorenz gauge was first detected for the Maxwell-Dirac system by d'Ancona, Foschi and Selberg \cite{AFS1}.

The paper \cite{ST} is the basis for our results. We show that in space dimensions $n=3$ and $n=2$
local well-posedness can also be proven for less regular data without finite energy, namely for $s > 3/4$ (for $n=2$ for a slightly different data space and $s \ge \frac{3}{4}$). We also consider the case $n \ge 4$ and prove local well-posedness for $ s > \frac{n}{2}-\frac{3}{4} $. These results rely on null conditions of most of the nonlinear terms. The necessary bilinear estimates in wave-Sobolev spaces were formulated in arbitrary dimension $n$ and proven for $n=2$ and $n=3$ by d'Ancona, Foschi and Selberg in \cite{AFS3} and \cite{AFS2}, respectively. We prove a special case in dimension $n \ge 4$ based mainly on a result by Klainerman and Tataru \cite{KT}.

We now formulate our main result. We assume the Lorenz condition
\begin{equation}
\label{6}
\partial^{\mu} A_{\mu} = 0
\end{equation}
and Cauchy data
\begin{align}
\label{7}
\phi(x,0) &= \phi_0(x) \in H^s \quad , \quad \partial_t \phi(x,0) = \phi_1(x) \in H^{s-1} \, , \\
\label{8}
F_{\mu \nu}(x,0)& = F^0_{\mu \nu}(x) \,\mbox{with}\, F^0_{\mu \nu} \in H^{s-1} \, \mbox{for} \, n=3 \,\mbox{and}\, D^{-\epsilon} F^0_{\mu \nu} \in H^{s-1+\epsilon} \,\mbox{for}\, n=2 \, , 
\end{align}
where $\epsilon$ is a small positive constant and
\begin{equation}
\label{9}
A_{\nu}(x,0) = a_{0 \nu}(x) \quad , \quad \partial_t A_{\nu}(x,0) = \dot{a}_{0 \nu}(x) \, ,
\end{equation} 
which fulfill the following conditions
\begin{equation}
\label{10}
a_{00} = \dot{a}_{00} = 0 \, , 
\end{equation}
\begin{equation}
\label{11}
\nabla a_{0j} \in H^{s-1} \, , \,  \dot{a}_{0j} \in H^{s-1}\, \mbox{for} \, n=3 \, ,
\end{equation}
\begin{equation}
\label{11'}
D^{1-\epsilon} a_{0j} \in H^{s-1+\epsilon} \, , \, D^{-\epsilon}\dot{a}_{0j} \in H^{s-1+\epsilon} \, \mbox{for} \, n=2 \, ,
\end{equation}
\begin{equation}
\label{12}
\partial^k a_{0k} = 0 \, ,
\end{equation}
\begin{equation}
\label{13}
\partial_j a_{0k} - \partial_k a_{0j} = F^0_{jk} \, ,
\end{equation}
\begin{equation}
\label{14}
\dot{a}_{0k} = F^0_{0k} \, , 
\end{equation}
\begin{equation}
\label{15}
\partial^k F^0_{0k}  = Im(\phi_0 \overline{\phi}_1) \, .
\end{equation}
(\ref{10}) can be assumed because otherwise the Lorenz condition does not determine the potential uniquely. (\ref{12}) follow from the Lorenz condition (\ref{6}) in connection with (\ref{10}). (\ref{13}) follows from (\ref{3}), similarly (\ref{14}) from (\ref{3}) and (\ref{10}). (\ref{1}) requires
$$ \partial^k F^0_{0k} = j_0(0) = Im(\phi_0 \overline{\phi}_1) + |\phi_0|^2 a_{00} = Im(\phi_0 \overline{\phi}_1) \, $$
thus (\ref{15}). By (\ref{12}) we have
$$ \Delta a_{0j} = \partial^k \partial_k a_{0j} = \partial^k(\partial^j a_{0k} - F^0_{jk}) = - \partial^k F^0_{jk} \, , $$
so that $a_{0j}$ is uniquely determined as
$$ a_{0j} = (-\Delta)^{-1} \partial^k F^0_{jk} $$
and fulfills (\ref{11}) and (\ref{11'}).

We define the wave-Sobolev spaces $X^{s,b}_{\pm}$ as the completion of the Schwarz space $\mathcal{S}({\mathbb R}^{n+1})$ with respect to the norm
$$ \|u\|_{X^{s,b}_{\pm}} = \| \langle \xi \rangle^s \langle  \tau \pm |\xi| \rangle^b \widehat{u}(\tau,\xi) \|_{L^2_{\tau \xi}}  $$
and $X^{s,b}_{\pm}[0,T]$ as the space of the restrictions to $[0,T] \times \mathbb{R}^n$.

We also define the spaces $H^{s,b}$ as the completion of  $\mathcal{S}({\mathbb R}^{n+1})$ with respect to the norm
$$ \|u\|_{H^{s,b}} =  \| \langle \xi \rangle^s \langle  |\tau| - |\xi| \rangle^b \widehat{u}(\tau,\xi) \|_{L^2_{\tau \xi}}  \, .$$

Let $\Lambda^{\alpha}$ , $D^{\alpha}$ , $\bar{D}_{\pm}^{\alpha}$ , $D_-^{\alpha}$ , $D_+^{\alpha}$ and $\Lambda_+^{\alpha}$ be the multipliers with symbols $\langle \xi \rangle^{\alpha}$, 
$|\xi|^{\alpha}$ , $\langle \tau \pm |\xi| \rangle^{\alpha}$ ,
 $ ||\tau|-|\xi||^{\alpha}$ , $ (|\tau |+|\xi|)^{\alpha}$ and $\langle |\tau|+|\xi| \rangle^{\alpha}$ , respectively, where $ \langle \, \cdot \, \rangle = (1+|\, \cdot \,|^2)^{\frac{1}{2}}$ .

$\Box = \partial_t^2 - \Delta$ is the d'Alembert operator. $a+ = a+\epsilon$ for a sufficiently small $\epsilon > 0$ . 

We also use the notation $ u \precsim v $ , if $|\widehat{u}| \lesssim \widehat{v}$ . \\[1em]
Our main theorem reads as follows:
\begin{theorem}
\label{Theorem1}
If $n \ge 3$ assume $s > \frac{n}{2}-\frac{3}{4} $ , $r > \frac{n}{2}-1$ and $s \ge r \ge s-1$ , $2r-s > \frac{n-3}{2} $ , $2s-r > \frac{n-1}{2}$ . 
If $n=2$ assume $ s \ge \frac{3}{4} $ , $r >\frac{1}{4}$ and $s \ge r \ge s-1$  , $2r-s > -\frac{1}{4}$ , $2s-r > \frac{3}{4}$ . \\
The data are assumed to fulfill (\ref{7}) - (\ref{15}). Then the problem (\ref{1}) - (\ref{6}) has a unique local solution
$$ \phi \in X_+^{s,\frac{1}{2}+}[0,T] +X_-^{s,\frac{1}{2}+}[0,T] \, , \, \partial_t \phi \in X_+^{s-1,\frac{1}{2}+}[0,T] +X_-^{s-1,\frac{1}{2}+}[0,T]$$
and
$$ F_{\mu \nu} \in X_+^{s-1,\frac{1}{2}+}[0,T] +X_-^{s-1,\frac{1}{2}+}[0,T]$$
in the case $n \ge 3$ and
$$ D^{-\epsilon} F_{\mu \nu} \in H^{s-1+\epsilon,\frac{1}{2}+}[0,T] \, , \, D^{-\epsilon} \partial_t F_{\mu \nu} \in H^{s-2+\epsilon,\frac{1}{2}+}[0,T]$$
in the case $n=2$ relative to a potential $A=(A_0,A_1,...,A_n)$, where
$A = A^{hom}_+ + A^{hom}_- + A^{inh}_+ + A^{inh}_-$ with
$D A^{hom}_{\pm} \in X^{r-1,1-\epsilon_0}_{\pm}[0,T]$ and $A^{inh}_{\pm} \in X^{r,1-\epsilon_0}_{\pm}[0,T]$ for $n=3$ and $D^{1-\epsilon} A^{hom}_{\pm} \in X^{r-1+\epsilon,\frac{3}{4}+}_{\pm}[0,T]$ and $D^{\epsilon} A^{inh}_{\pm} \in X^{r-\epsilon,\frac{3}{4}+}_{\pm}[0,T]$ for $n=2$, where $\epsilon$ is a small positive number. 
\end{theorem}
\noindent{\bf Remarks:} 
\begin{enumerate}
\item We immediately obtain $\phi \in C^0([0,T],H^s) \cap C^1([0,T],H^{s-1}) $ , and $F_{\mu \nu} \in C^0([0,T],H^{s-1}) \cap C^1([0,T],H^{s-2}) $ for $n=3$ and $D^{-\epsilon} F_{\mu \nu} \in C^0([0,T],$ $H^{s-1+\epsilon}) \cap C^1([0,T],H^{s-2+\epsilon})$ for $n=2$ . 
\item The case $  \frac{n}{2}-\frac{3}{4}+\delta \, , \, r = \frac{n}{2}-1+\delta $ 
for $n \ge 3$ and $s= \frac{3}{4}+\delta \, , \, r = \frac{1}{4} + \delta $ for $n=2$ is admissible, where $\delta > 0$ is an arbitrary number.
\end{enumerate}

We can reformulate the system (\ref{1}),(\ref{2}) under the Lorenz condition (\ref{6}) as follows:
$$\square A_{\mu} = \partial^{\nu} \partial_{\nu} A_{\mu} = \partial^{\nu}(\partial^{\mu} A_{\nu} - F_{\mu \nu}) = -\partial^{\nu} F_{\mu \nu} = - j_{\mu} \, , $$
thus (using the notation $\partial = (\partial_0,\partial_1,...,\partial_n)$):
\begin{equation}
\label{16}
\square A = -Im (\phi \overline{\partial \phi}) - A|\phi|^2 =: N(A,\phi) 
\end{equation}
and
\begin{align*}
m^2 \phi & = D^{(A)}_{\mu} D^{(A)\mu} \phi = \partial_{\mu} \partial^{\mu} \phi -iA_{\mu} \partial^{\mu} \phi -i\partial_{\mu}(A^{\mu} \phi) - A_{\mu}A^{\mu} \phi \\
& = \square \phi - 2i A^{\mu} \partial_{\mu} \phi - A_{\mu} A^{\mu} \phi \,
\end{align*}
thus
\begin{equation}
\label{17}
(\square -m^2) \phi = 2i A^{\mu} \partial_{\mu} \phi + A_{\mu} A^{\mu} \phi =: M(A,\phi) \, .
\end{equation}
Conversely, if $\square A_{\mu} = -j_{\mu}$ and $F_{\mu \nu} := \partial_{\mu} A_{\nu} - \partial_{\nu} A_{\mu}$ and the Lorenz condition (\ref{6}) holds then
$$ \partial^{\nu} F_{\mu \nu} = \partial^{\nu}(\partial_{\mu} A_{\nu} - \partial_{\nu} A_{\mu}) = \partial_{\mu} \partial^{\nu} A_{\nu} - \partial^{\nu} \partial_{\nu} A_{\mu} = -\square A_{\mu} = j_{\mu} \, $$
thus (\ref{1}),(\ref{2}) is equivalent to (\ref{16}),(\ref{17}), if (\ref{3}),(\ref{4}) and (\ref{6}) are satisfied.

The paper is organized as follows: in chapter 2 we prove the null structure of $A^{\mu} \partial_{\mu} \phi$ and in the Maxwell part. In chapter 3 the local well-posedness result for (\ref{16}),(\ref{17}) is formulated (Theorem \ref{Theorem2}). This relies on the null structure of $A^{\mu} \partial_{\mu} \phi$ and the bilinear estimates in wave-Sobolev spaces by d'Ancona, Foschi and Selberg (\cite{AFS3} and \cite{AFS2}), which are given in chapter 4. In chapter 5 we prove Theorem \ref{Theorem2}. In chapter 6 we prove our main theorem (Theorem \ref{Theorem1}). We show that the Maxwell-Klein-Gordon system is satisfied and $F_{\mu \nu}$ has the desired regularity properties using the null structure of the Maxwell part and again the bilinear estimates by d'Ancona, Foschi and Selberg.

\section{Null structure}
\noindent{\bf Null structure of $A^{\mu} \partial_{\mu}\phi$.} \\
Using the Riesz transform $R_k := D^{-1} \partial_k $ and $\mathbf{A} = (A_1,...,A_n)$ we use the decomposition 
$$ \mathbf{A} = A^{df} + A^{cf} \, ,$$
where
$$ A^{df}_j = R^k(R_j A_k - R_k A_j) \quad , \quad A^{cf}_j = - R_j R_k A^k \, , $$
so that
$$ A^{\mu} \partial_{\mu} \phi  = P_1 + P_2 \, , $$
with
\begin{align*}
 P_1 &= -A^0 \partial_t \phi + A^{cf} \cdot \nabla \phi \\
& = -A^0 \partial_t \phi - D^{-2} \partial_j \partial_k A^k \partial^j \phi\\
& = -A^0 \partial_t \phi - D^{-2}  \nabla \partial_t A^0 \cdot \nabla \phi \, ,
\end{align*}
where we used the Lorenz gauge $\partial_k A^k = \partial_t A^0$ ,
and 
$$ P_2 = A^{df} \cdot \nabla \phi \, . $$
Now define
$$ Q_{jk}(\phi,\psi) := \partial_j \phi \partial_k \psi - \partial_k \phi \partial_j \psi \, , $$
so that
\begin{align*}
&\sum_{j,k} Q_{jk}(D^{-1}(R_j A_k - R_k A_j),\phi) \\
& = \sum_{j,k} [ \partial_j(D^{-1}(R_j A_k - R_k A_j))\partial_k \phi -  \partial_k(D^{-1}(R_j A_k - R_k A_j)) \partial_j \phi] \\
& = \sum_{j,k} [ D^{-2}(\partial_j^2 A_k - \partial_j \partial_k A_j) \partial_k \phi - D^{-2}(\partial_k \partial_j A_k - \partial_k^2 A_j) \partial_j \phi] \\
&=-2\big(\sum_{j,k} D^{-2} \partial_k \partial_j A_k \partial_j \phi + \sum_j  A_j \partial_j \phi \big) \\
& = -2 \big(\sum_{j,k} R_k R_j A_k  \partial_j \phi- \sum_{j,k}  R_k R_k A_j \partial_j \phi \big)\\
& = -2P_2 \, .
\end{align*}
Thus the symbol $p_2$ of $P_2$ fulfills
\begin{equation}
\label{2.0}
p_2(\eta,\xi) = \frac{1}{2} \sum_{j,k} \left|\frac{1}{|\eta|^2}(\eta_k \xi_j - \eta_j \xi_k) (\eta_j - \eta_k) \right| \lesssim \sum_{j,k} \frac{|\eta_k \xi_j - \eta_j \xi_k|}{|\eta|} \lesssim |\xi| \angle(\eta,\xi) \, ,
\end{equation}
where $\angle(\eta,\xi)$ denotes the angle between $\eta$ and $\xi$.

Before we consider $P_1$ we define
$$ A_{\pm} := \frac{1}{2}(A \pm (iD)^{-1} A_t) \, , $$
so that $A=A_+ + A_-$ and $A_t = iD(A_+-A_-)$, and
$$\phi_{\pm} := \frac{1}{2}(\phi \pm (i \Lambda_m)^{-1} \phi_t) $$
with $\Lambda_m := (m^2-\Delta)^{\frac{1}{2}}$, so that $\phi = \phi_+ + \phi_-$ and $\phi_t = i \Lambda_m (\phi_+ - \phi_-)$. 

We transform (\ref{16}),(\ref{17}) into
\begin{align}
\label{2.1}
(i\partial_t \pm \Lambda_m) \phi_{\pm} & = -(\pm 2 \Lambda_m)^{-1} M(A,\phi) \\
\label{2.2}
(i\partial_t \pm D) A_{\pm} & = -(\pm 2 D)^{-1} N(A,\phi) \, .
\end{align}
Then we obtain
\begin{align}
\label{2.2a}
iP_1 & = (A_{0+}+A_{0-})\Lambda (\phi_+ - \phi_-) + D^{-1} \nabla(A_{0+} - A_{0-}) \cdot \nabla(\phi_+ + \phi_-) \\
\nonumber
& = \sum_{\pm_1,\pm_2} \pm_2 \mathcal{A}_{(\pm_1,\pm_2)} (A_{0 \pm1},\phi_{\pm 2}) \, ,
\end{align}
where
$$ \mathcal{A}_{(\pm_1,\pm_2)}(f,g) := f \Lambda_m g + D^{-1} \nabla(\pm_1 f) \cdot \nabla(\pm_2 g) \, . $$
Its symbol $a_{(\pm_1,\pm_2)} (\eta,\xi)$ is bounded by the elementary estimate (\cite{ST}, Lemma 3.1):
\begin{equation}
\label{2.3}
|a_{(\pm_1,\pm_2)} (\eta,\xi) | \sim \left|\langle \xi \rangle_m - \frac{(\pm_1 \eta) \cdot (\pm_2 \xi)}{|\eta|}\right| \lesssim m + |\xi| \angle(\pm_1 \eta, \pm_2 \xi) \,,
\end{equation}
where $ \langle \xi \rangle_m = (m^2 + |\xi|^2)^{\frac{1}{2}}$ .

We now use the well-known
\begin{lemma}
\label{Lemma} Assume $0 \le \alpha \le \frac{1}{2}$.
For arbitrary signs $\pm,\pm_1,\pm_2$ , $\lambda,\mu \in \mathbb{R}$ and $\eta,\xi \in {\mathbb R}^n$ the following estimate holds:
\begin{equation}
\label{angle1}
\angle (\pm_1 \eta,\pm_2 \xi)  \lesssim \frac{\langle (\lambda+\mu) \pm |\eta+\xi| \rangle^{\frac{1}{2}-\alpha}}{\min(\langle\eta\rangle, \langle \xi \rangle)^{\frac{1}{2}-\alpha}}   + \frac{ \langle \lambda \pm_1 |\eta| \rangle^{\frac{1}{2}} 
+ \langle  \mu \pm_2 |\xi |\rangle^{\frac{1}{2}}}{\min(\langle \eta \rangle,\langle \xi \rangle)^{\frac{1}{2}}}
\end{equation}
In the case of different signs $\pm_1$ and $\pm_2$ the following (improved) estimate holds:
\begin{align}
\nonumber
\angle (\pm_1 \eta,\pm_2 \xi) & \lesssim \frac{|\eta + \xi|^{\frac{1}{2}}}{|\eta|^{\frac{1}{2}} | \xi|^{\frac{1}{2}}} \big( \min(|\eta|,|\xi|)^{\alpha} \langle (\lambda + \mu) \pm |\eta + \xi| \rangle^{\frac{1}{2}-\alpha} + \langle \lambda \pm_1 |\eta| \rangle^{\frac{1}{2}} \\
\nonumber
& \hspace{5em} + \langle \mu \pm_2 |\xi |\rangle^{\frac{1}{2}}\big) \\
\label{angle2}
& \lesssim \frac{|\eta + \xi|^{\frac{1}{2}}}{|\eta|^{\frac{1}{2}} |\xi|^{\frac{1}{2}}} \min(|\eta|,| \xi|)^{\alpha} \langle (\lambda + \mu) \pm |\eta + \xi| \rangle^{\frac{1}{2}-\alpha} \\ \nonumber
&
\hspace{5em} + \frac{\langle \lambda \pm_1 |\eta| \rangle^{\frac{1}{2}} + \langle \mu \pm_2 |\xi|\rangle^{\frac{1}{2}}}{\min(|\eta|,|\xi|)^{\frac{1}{2}}} \, .
\end{align}
\end{lemma}
\begin{proof} These results follow from \cite{AFS}, Lemma 7 and the considerations ahead of and after that lemma, where we use that in the case of different signs $\pm_1$ and $\pm_2$ we have
$\angle (\pm_1 \eta,\pm_2 \xi) =  \angle (\eta,-\xi)$. Cf. also \cite{ST}, Lemma 4.3.
\end{proof}

Using the estimates for the symbols above and (\ref{angle1}) we summarize our results as follows:
\begin{align}
\label{Q1}
A^{\mu} \partial_{\mu} \phi  \precsim \sum_{\pm_1,\pm_2} & \Big( \bar{D}_{\pm}^{\half-2\epsilon} (\Lambda^{-\half+2\epsilon} A_{\pm_1} D \phi_{\pm_2}) + \bar{D}_{\pm}^{\half-2\epsilon}(A_{\pm_1} D^{\half + 2\epsilon} \phi_{\pm_2}) \\ \nonumber
& + \Lambda^{-\half} \bar{D}_{\pm_1}^{\half} A_{\pm_1} D \phi_{\pm_2} +\bar{D}_{\pm_1}^{\half} A_{\pm_1} D^{\half} \phi_{\pm_2} \\ \nonumber
& + \Lambda^{-\half} A_{\pm_1}\bar{D}_{\pm_2}^{\half} D \phi_{\pm_2} + A_{\pm_1} D^{\half}\bar{D}_{\pm_2}^{\half} \phi_{\pm_2} \Big) + \sum_{\pm_1,\pm_2} A_{\pm_1} \phi_{\pm_2} \, .
\end{align}
\\[0.5em]
{\bf Null structure in the Maxwell part.}\\
We start from Maxwell's equations (\ref{1}), i.e. $-\partial^0 F_{l0} + \partial^k F_{lk} = j_l$ and $\partial^k F_{0k} = j_0$ and obtain
\begin{align}
\nonumber
\square F_{k0} & = -\partial_0(\partial_0 F_{k0}) + \partial^l \partial_l F_{k0} \\ 
\nonumber
& = -\partial_0(\partial^l F_{kl} - j_k) + \partial^l \partial_l F_{k0} \\
\nonumber
& = -\partial^l \partial_0(\partial_k A_l - \partial_l A_k) + \partial_0 j_k + \partial^l \partial_l F_{k0} \\
\nonumber
& = -\partial^l[\partial_k(\partial_0 A_l - \partial_l A_0) - \partial_l(\partial_0 A_k - \partial_k A_0)] + \partial_0 j_k + \partial^l \partial_l F_{k0} \\
\nonumber
& = -\partial^l \partial_k F_{0l} + \partial^l \partial_l F_{0k} + \partial_0 j_k + \partial^l \partial_l F_{k0} \\
\nonumber
& =  -\partial^l \partial_k F_{0l} +\partial_0 j_k \\
\label{2.10}
& = -\partial_k j_0 + \partial_0 j_k
\end{align}
and
\begin{align}
\nonumber
\square F_{kl} & = -\partial_0 \partial_0 F_{kl} + \partial^m \partial_m F_{kl} \\\nonumber
& = -\partial_0 \partial_0 (\partial_k A_l - \partial_l A_k) + \partial^m \partial_m F_{kl} \\\nonumber
& = -\partial_0 \partial_k(\partial_0 A_l - \partial_l A_0) + \partial_0 \partial_l (\partial_0 A_k - \partial_k A_0) + \partial^m \partial_m F_{kl} \\\nonumber
& = -\partial_0 \partial_k F_{0l} + \partial_0 \partial_l F_{0k} + \partial^m \partial_m F_{kl} \\\nonumber
& = \partial_k \partial_0 F_{l0} - \partial_l \partial_0 F_{k0} + \partial^m \partial_m F_{kl} \\\nonumber
& = \partial_k (\partial^m F_{lm} - j_l) - \partial_l(\partial^m F_{km} - j_k)+ \partial^m \partial_m F_{kl} \\\nonumber
& = \partial_k \partial^m F_{lm} - \partial_l \partial^m F_{km} + \partial^m \partial_m F_{kl} +\partial_l j_k - \partial_k j_l \\\nonumber
& = \partial_k \partial^m(\partial_l A_m - \partial_m A_l) - \partial_l \partial^m(\partial_k A_m - \partial_m A_k)+ \partial^m \partial_m F_{kl} +\partial_l j_k - \partial_k j_l \\\nonumber
& = \partial^m \partial_m (\partial_l A_k - \partial_k A_l) + \partial^m \partial_m F_{kl} +\partial_l j_k - \partial_k j_l \\\nonumber
& = \partial^m \partial_m F_{lk} + \partial^m \partial_m F_{kl} +\partial_l j_k - \partial_k j_l \\
\label{2.11}
&   = \partial_l j_k - \partial_k j_l \, .
\end{align}
By the definition ({5}) of $j_{\mu}$ we obtain
\begin{align}
\nonumber
\partial_0 j_k - \partial_k j_0 
 =& \, Im(\partial_0 \phi \overline{\partial_k \phi}) + Im(\phi \overline{\partial_0 \partial_k \phi}) + \partial_0(A_k |\phi|^2) \\
 \nonumber
&- Im(\partial_k \phi \overline{\partial_0 \phi}) - Im(\phi \overline{\partial_k \partial_0 \phi}) - \partial_k(A_0 |\phi|^2) \\
\label{2.10a}
 =&\, Im(\partial_t \phi\overline{\partial_k \phi} - \partial_k \phi \overline{\partial_t \phi}) + \partial_t(A_k |\phi|^2) - \partial_k(A_0 |\phi|^2)
\end{align}
and
\begin{align}
\nonumber
\partial_l j_k - \partial_k j_l 
 =& \, Im(\partial_l \phi \overline{\partial_k \phi}) + Im(\phi \overline{\partial_l \partial_k \phi}) + \partial_l(A_k |\phi|^2) \\
 \nonumber
&- Im(\partial_k \phi \overline{\partial_l \phi}) - Im(\phi \overline{\partial_k \partial_l \phi}) - \partial_k(A_l |\phi|^2) \\
\label{2.11b}
 =&\, Im(\partial_l \phi\overline{\partial_k \phi} - \partial_k \phi \overline{\partial_l \phi}) + \partial_l(A_k |\phi|^2) - \partial_k(A_l |\phi|^2) \, .
\end{align}
After the decomposition $\phi = \phi_+ + \phi_-$ we have
\begin{equation}
\label{2.11c}
\partial_l \phi\overline{\partial_k \phi} - \partial_k \phi \overline{\partial_l \phi} = \sum_{\pm_1,\pm_2} \mathcal{C}_{\pm_1,\pm_2} (\phi_{\pm_1,\phi_{\pm_2}})                            \, , \end{equation}
where
\begin{equation}
\label{2.12}
\mathcal{C}_{\pm_1,\pm_2} (f,g) := \partial_l f \overline{\partial_k g} - \partial_k f  \overline{\partial_l g} \, .
\end{equation}
Its symbol
\begin{equation}
c_{\pm_1,\pm_2}(\eta,\xi) = \eta_l \xi_k - \eta_k \xi_l
\label{2.13}
\end{equation}
fulfills
\begin{equation}
\label{2.14}
|c_{\pm_1,\pm_2} (\eta,\xi) = |(\pm_1 \eta_l)(\pm_2 \xi_k) - (\pm_1 \eta_k)(\pm_2 \xi_l)| \lesssim |\eta||\xi| \angle(\pm_1 \eta,\pm_2 \xi) \, . 
\end{equation}
Similarly using $\partial_t \phi = i \Lambda_m (\phi_+ - \phi_-)$ we have
\begin{equation}
\label{2.10b}
\partial_t \phi\overline{\nabla \phi} - \nabla \phi \overline{\partial_t \phi} = \sum_{\pm_1,\pm_2} (\pm_1 1)(\pm_2 1)\mathcal{B}_{\pm_1,\pm_2} (\phi_{\pm_1,\phi_{\pm_2}})                            \, , \end{equation}
where
\begin{equation}
\label{2.15}
\mathcal{B}_{\pm_1,\pm_2} (f,g) := i(\Lambda_m f \overline{\nabla (\pm_2 g)} - \nabla (\pm_1 f)  \overline{\Lambda_m g}) \, .
\end{equation}
Its symbol
\begin{equation}
b_{\pm_1,\pm_2}(\eta,\xi) = \langle \eta \rangle_m (\pm_2 \xi) - \langle\xi \rangle_m (\pm_1 \eta)
\label{2.16}
\end{equation}
can be estimated elementarily (\cite{ST}, Lemma 3.2):
\begin{equation}
\label{2.17}
|b_{\pm_1,\pm_2} (\eta,\xi)| \lesssim m(|\eta| + |\xi|) + |\eta||\xi| \angle(\pm_1 \eta,\pm_2 \xi) \, . 
\end{equation}
\\
Combining these estimates for the symbols and (\ref{angle1}) we obtain
\begin{align}
\label{Q2}
& Im(\partial_t \phi\overline{\partial_k \phi} - \partial_k \phi \overline{\partial_t \phi}) +  Im(\partial_l \phi\overline{\partial_k \phi} - \partial_k \phi \overline{\partial_l \phi}) \\
\nonumber
& \precsim \sum_{\pm,\pm_1,\pm_2} \left( \bar{D}_{\pm}^{\half-2\epsilon} (D^{\half+2\epsilon} \phi_{\pm_1} D \phi_{\pm_2}) + D^{\half} \bar{D}_{\pm_1}^{\half} \phi_{\pm_1} D \phi_{\pm_2} +  D^{\half} \phi_{\pm_1}D \bar{D}_{\pm_2}^{\half} \phi_{\pm_2} \right) \\
\nonumber
&  \hspace{1em} +
\sum_{\pm_1,\pm_2} \phi_{\pm_1} D \phi_{\pm_2} \, .
\end{align}
For different signs $ \pm_1 $ and $\pm_2$ we use (\ref{angle2}) and obtain the bound
\begin{align}
\label{Q3}
& Im(\partial_t \phi\overline{\partial_k \phi} - \partial_k \phi \overline{\partial_t \phi}) +  Im(\partial_l \phi\overline{\partial_k \phi} - \partial_k \phi \overline{\partial_l \phi}) \\
\nonumber
& \precsim \sum_{\pm,\pm_1,\pm_2} \Big( \bar{D}_{\pm}^{\half-2\epsilon} D^{\half} (D^{\half+\epsilon} \phi_{\pm_1} D^{\half+\epsilon} \phi_{\pm_2}) + D^{\half} \bar{D}_{\pm_1}^{\half} \phi_{\pm_1} D \phi_{\pm_2} \\ \nonumber
& \hspace{5em }+ D^{\half} \phi_{\pm_1} D \bar{D}_{\pm_2}^{\half} \phi_{\pm_2} \Big)  +
\sum_{\pm_1,\pm_2} \phi_{\pm_1} D \phi_{\pm_2} \, .
\end{align}

\section{Local well-posedness}
Recall $\phi_{\pm}=\frac{1}{2}(\phi \pm (i\Lambda_m)^{-1}\phi_t)$ , so that $\phi=\phi_+ + \phi_-$ and $ \partial_t \phi = i \Lambda_m(\phi_+ - \phi_-)$, and $A_{\pm} = \frac{1}{2}(A \pm (iD)^{-1} A_t)$ so that $A=A_+ + A_-$ and $\partial_t A = iD(A_+-A_-)$, we write (\ref{2.1}),(\ref{2.2}) as follows:
\begin{align}
\label{3.1}
(i\partial_t \pm \Lambda_m) \phi_{\pm} & = - (\pm 2 \Lambda_m)^{-1} \mathcal{M}(\phi_+,\phi_-,A_+,A_-) \\
\label{3.2}
(i\partial_t \pm D)A_{\pm} & = -(\pm 2D)^{-1} \mathcal{N}(\phi_+,\phi_-,A_+,A_-) \, ,
\end{align}
where
\begin{align}
\label{3.3}
\mathcal{M}(\phi_+,\phi_-,A_+,A_-) & = A^{\mu} \partial_{\mu} \phi + A_{\mu} A^{\mu} \phi \\
\label{3.4}
\mathcal{N}_0(\phi_+,\phi_-,A_+,A_-) & = Im(\phi i \Lambda_m(\overline{\phi}_+ - \overline{\phi}_-)) - A_0 |\phi|^2 \\
\label{3.5}
\mathcal{N}_j(\phi_+,\phi_-,A_+,A_-) & = -Im(\phi \overline{\partial_j \phi}) -A_j |\phi|^2 \, .
\end{align}
The initial data are
\begin{align}
\label{3.5a}
\phi_{\pm}(0) & = \frac{1}{2}(\phi_0 \pm (i \Lambda_m)^{-1} \phi_1) \\
\label{3.5b}
A_{0\pm}(0)& = \frac{1}{2}(a_{00}\pm(iD^{-1})\dot{a}_{00}) = 0 \\
\label{3.5c}
A_{j\pm}(0) & = \frac{1}{2}(a_{0j} \pm (iD)^{-1} \dot{a}_{0j}) \, .
\end{align}
(\ref{3.5b}) follows from (\ref{10}). From (\ref{7}) we have $\phi_{\pm}(0) \in H^s$, and from (\ref{11}) we have for $r\le s$ in the case $n=3$: $\nabla a_{0j} \in H^{r-1}$ , $\dot{a}_{0j} \in H^{r-1}$, so that $\nabla A_{j\pm}(0) \in H^{r-1}$, whereas in the case $n=2$ we have $D^{1-\epsilon} a_{0j} \in H^{r-1+\epsilon}$ , $D^{-\epsilon} \dot{a}_{0j} \in H^{r-1+\epsilon}$, so that $D^{1-\epsilon} A_{j\pm}(0) \in H^{r-1+\epsilon}$.

We split $A_{\pm} = A_{\pm}^{hom} + A_{\pm}^{inh}$ into its homogeneous and inhomogeneous part, where $(i\partial_t \pm D)A_{\pm}^{hom} =0$ with data as in (\ref{3.5b}) and (\ref{3.5c}) and $A_{\pm}^{inh}$ is the solution of (\ref{3.2}) with zero data. By the linear theory we obtain for $b>1/2$:
$$ \|\phi_{\pm}^{hom}\|_{X^{s,b}_{\pm}} \lesssim \|\phi_{\pm}(0)\|_{H^s}$$
and for $\beta > 1/2$:
$$ \| D A_{\pm}^{hom} \|_{X_{\pm}^{r-1,\beta}} \lesssim \| D A_{\pm}(0)\|_{H^{r-1}} \quad \mbox{for} \, n=3 \, $$
$$\| D^{1-\epsilon} A_{\pm}^{hom} \|_{X_{\pm}^{r-1+\epsilon,\beta}} \lesssim \| D^{1-\epsilon} A_{\pm}(0)\|_{H^{r-1+\epsilon}} \quad \mbox{for} \, n=2 \, . $$

Our aim is to show the following local well-posedness result:
\begin{theorem}
\label{Theorem2}
1. If $n \ge 3$ assume 
$$ s > \frac{n}{2}-\frac{3}{4} \, , \, r > \frac{n}{2}-1 \, , \, s \ge r-1 \, , \, r \ge s-1 \, , \,2r-s > \frac{n-3}{2} \, , \,2s-r > \frac{n-1}{2} \,.$$
 Let $\phi_{\pm}(0) \in H^s$ and $D A_{\pm}(0) \in H^{r-1}$  be given. Then the system (\ref{3.1}),(\ref{3.2}) has a unique local solution
$$\phi_{\pm} \in X_{\pm}^{s,\frac{1}{2}+}[0,T] \, , \,D A_{\pm}^{hom} \in X_{\pm}^{r-1,1-\epsilon_0}[0,T] \, , \,A_{\pm}^{inh} \in X_{\pm}^{r,1-\epsilon_0}[0,T] \, ,$$  where $\epsilon_0 >0$ is sufficiently small. \\
2. If $n = 2$ assume 
$$ s \ge \frac{3}{4} \, , \, r > \frac{1}{4}\, , \, s \ge r-1 \, , \,r \ge s-1 \, , \, 2r-s > -\frac{1}{4} \,  , \, 2s-r > \frac{3}{4}\, .$$ 
 Let $\phi_{\pm}(0) \in H^s$ and $D^{1-\epsilon} A_{\pm}(0) \in H^{r-1+\epsilon}$  be given. Then the system (\ref{3.1}),(\ref{3.2}) has a unique local solution 
$$\phi_{\pm} \in X_{\pm}^{s,\frac{1}{2}+}[0,T] \, , \, D^{1-\epsilon} A_{\pm}^{hom} \in X_{\pm}^{r-1+\epsilon,\frac{3}{4}+}[0,T] \, ,  \, D^{\epsilon} A_{\pm}^{inh} \in X_{\pm}^{r-\epsilon,\frac{3}{4}}[0,T] \, , $$
where $\epsilon >0$ is sufficiently small.
\end{theorem}

\section{Preliminaries}

Fundamental for the proof of Theorem \ref{Theorem2} are the following bilinear estimates in wave-Sobolev spaces which were proven by d'Ancona, Foschi and Selberg in the cases $n=2$ in \cite{AFS2} and $n=3$ in \cite{AFS3} in a more general form which include many limit cases which we do not need.
\begin{theorem}
\label{Theorem3}
Let $n=2$ or $n=3$. The estimates
$$\|uv\|_{X_{\pm}^{-s_0,-b_0}} \lesssim \|u\|_{X^{s_1,b_1}_{\pm_1}} \|v\|_{X^{s_2,b_2}_{\pm_2}} $$
and 
$$\|uv\|_{H^{-s_0,-b_0}} \lesssim \|u\|_{H^{s_1,b_1}} \|v\|_{XH{s_2,b_2}} $$
hold, provided the following conditions hold:
\begin{align*}
\nonumber
& b_0,b_1,b_2 \ge 0 \\
\nonumber
& b_0 + b_1 + b_2 > \frac{1}{2} \\
\nonumber
& b_0 + b_1 > 0 \\
\nonumber
& b_0 + b_2 > 0 \\
\nonumber
& b_1 + b_2 > 0 \\
\nonumber
&s_0+s_1+s_2 > \frac{n+1}{2} -(b_0+b_1+b_2) \\
\nonumber
&s_0+s_1+s_2 > \frac{n}{2} -(b_0+b_1) \\
\nonumber
&s_0+s_1+s_2 > \frac{n}{2} -(b_0+b_2) \\
\nonumber
 &s_0+s_1+s_2 > \frac{n}{2} -(b_1+b_2)\\
\end{align*}
\begin{align}
\nonumber
&s_0+s_1+s_2 > \frac{n-1}{2} - b_0 \\
\nonumber
&s_0+s_1+s_2 > \frac{n-1}{2} - b_1 \\
\nonumber
&s_0+s_1+s_2 > \frac{n-1}{2} - b_2 \\
\nonumber
&s_0+s_1+s_2 > \frac{n+1}{4} \\
 \label{5.13}
&(s_0 + b_0) +2s_1 + 2s_2 > \frac{n}{2} \\
\nonumber
&2s_0+(s_1+b_1)+2s_2 > \frac{n}{2} \\
\nonumber
&2s_0+2s_1+(s_2+b_2) > \frac{n}{2} \\
\nonumber
&s_1 + s_2 \ge 0 \\
\nonumber
&s_0 + s_2 \ge 0 \\
\nonumber
&s_0 + s_1 \ge 0 \, .
\end{align}
If $n=3$ the condition (\ref{5.13}) is only necessary in the case when $\langle \xi_0 \rangle \lesssim \langle \xi_1 \rangle \sim \langle \xi_2 \rangle$ and also $\pm_1$ and $\pm_2$ are different signs. Here $\xi_0$,$\xi_1$ and $\xi_2$ denote the spatial frequencies of $\widehat{uv}$,$\widehat{u}$ and $\widehat{v}$, respectively.
\end{theorem}

Next, we want to prove a special case of these bilinear estimate, which holds in higher dimensions.
We start by recalling the Strichartz type estimates for the wave equation.
\begin{prop}
\label{Str}
If $n \ge 2$ and
\begin{equation}
2 \le q \le \infty, \quad
2 \le r < \infty, \quad
\frac{2}{q} \le (n-1) \left(\half - \frac{1}{r} \right), 
\end{equation}
then the following estimate holds:
$$ \|u\|_{L_t^q L_x^r} \lesssim \|u\|_{H^{\frac{n}{2}-\frac{n}{r}-\frac{1}{q},\half+}} \, .$$
\end{prop}
\begin{proof}
A proof can be found for e.g. in \cite{GV}, Prop. 2.1, which is combined with the transfer principle.
\end{proof}
The following proposition was proven by \cite{KT}.
\begin{prop}
\label{TheoremE}
Let $n \ge 2$, and let $(q,r)$ satisfy:
$$
2 \le q \le \infty, \quad
2 \le r < \infty, \quad
\frac{2}{q} \le (n-1) \left(\half - \frac{1}{r} \right)
$$
Assume that
\begin{gather*}
0 < \sigma < n - \frac{2n}{r} - \frac{4}{q}, \\
s_1, s_2 < \frac{n}{2} - \frac{n}{r} - \frac{1}{q}, \\
s_1 + s_2 + \sigma = n - \frac{2n}{r} - \frac{2}{q}. \end{gather*} 
then
$$
\|D^{-\sigma}(uv)\|_{L_t^{q/2} L_x^{r/2}}
\lesssim
\|u\|_{H^{s_1,\half+}}
\|v\|_{H^{s_2,\half+}} \, .
$$
\end{prop}

The following product estimate for wave-Sobolev spaces is a special case of the very convenient much more general atlas formulated by \cite{AFS} in arbitrary dimension, but proven only in the case $1\le n \le 3$. (\cite{AFS} and \cite{AFS1}). Therefore we have to give a proof.

\begin{prop}
\label{Prop.3.6}
Assume $n \ge 3$ and
$$s_0+s_1+s_2 > \frac{n-1}{2} \, , \, (s_0+s_1+s_2)+s_1+s_2 > \frac{n}{2} \, , \, s_0+s_1 \ge 0 \, , \, s_0+s_2 \ge 0 \, , \, s_1+s_2 \ge 0 \, . $$
The following estimate holds:
$$ \|uv\|_{H^{-s_0,0}} \lesssim \|u\|_{H^{s_1,\half+}} \|v\|_{H^{s_2,\half+}} \, . $$
\end{prop}
\begin{proof}
We only consider the case $n \ge 4$ , because the case $n=3$ follows similarly and is contained in Theorem \ref{Theorem3}.
We have to prove
\begin{align*}
I & := \int_* \frac{\widehat{u}_1(\xi_1,\tau_1)}{\angles{\xi_1}^{s_1} \angles{|\xi_1|-|\tau_1|}^{\half+}}  \frac{\widehat{u}_2(\xi_2,\tau_2)}{\angles{\xi_2}^{s_2} \angles{|\xi_2|-|\tau_2|}^{\half+}}  \frac{\widehat{u}_0(\xi_0,\tau_0)}{\angles{\xi_0}^{s_0}} \lesssim \|u_1\|_{L^2_{xt}}
\|u_2\|_{L^2_{xt}} \, .
\end{align*}
Here * denotes integration over $\xi_0+\xi_1+\xi_2=0 $ and $\tau_0+\tau_1+\tau_2=0$ . Remark, that we may assume that the Fourier transforms are nonnegative. We consider different regions. \\
1. If $|\xi_0| \sim |\xi_1| \gtrsim |\xi_2|$ and $s_2 \ge 0$, we obtain
\begin{align*}
I \sim \int_* \frac{\widehat{u}_1(\xi_1,\tau_1)}{\angles{\xi_1}^{s_1+s_0} \angles{|\xi_1|-|\tau_1|}^{\half+}}  \frac{\widehat{u}_2(\xi_2,\tau_2)}{\angles{\xi_2}^{s_2} \angles{|\xi_2|-|\tau_2|}^{\half+}}  \widehat{u}_0(\xi_0,\tau_0) \, .
\end{align*}
Thus we have to show
$$ \|uv\|_{L^2_{xt}} \lesssim \|u\|_{H^{s_1+s_0,\half+}} \|v\|_{H^{s_1,\half+}} \, . $$
By Prop. \ref{Str} we obtain
$$ \|uv\|_{L^2_{xt}} \lesssim \|u\|_{L^{\infty}_t L^2_x} \|v\|_{L^2_t L^{\infty}_x} \lesssim \|u\|_{H^{0,\half+}} \|v\|_{H^{\frac{n-1}{2}+,\half+}} $$
and also
$$ \|uv\|_{L^2_{xt}} \lesssim  \|u\|_{H^{\frac{n-1}{2}+,\half+}}      \|v\|_{H^{0,\half+}} \, .$$
Bilinear interpolation gives for $0 \le \theta \le 1 $ :
$$ \|uv\|_{L^2_{xt}} \lesssim \|u\|_{H^{\frac{n-1}{2}(1-\theta)+,\half+}} \|v\|_{H^{\frac{n-1}{2}\theta+,\half+}}  \, , $$
so that 
$$   \|uv\|_{L^2_{xt}} \lesssim \|u\|_{H^{s_1+s_0,\half+}} \|v\|_{H^{s_2,\half+}} \, , $$
if $ s_0+s_1+s_2 > \frac{n-1}{2} $ and $s_1+s_0 \ge 0$ .\\
2. If $|\xi_0| \sim |\xi_2| \gtrsim |\xi_1|$ and $s_1 \ge 0$ , we obtain similarly
$$  \|uv\|_{L^2_{xt}} \lesssim \|u\|_{H^{s_2+s_0,\half+}} \|v\|_{H^{s_1,\half+}} \, , $$
if $ s_0+s_1+s_2 > \frac{n-1}{2} $ and $s_2+s_0 \ge 0$ .\\
3. If $|\xi_1| \ge |\xi_2|$ , $s_0 \le 0$ and $s_2 \ge 0$, we have $|\xi_0| \lesssim |\xi_1| $ , so that
$ \angles{\xi_0}^{-s_0} \lesssim \angles{\xi_1}^{-s_0}$ and we obviously obtain the same result as in 1. \\
4. If $|\xi_1| \le |\xi_2|$ , $s_0 \le 0$ and $s_1 \ge 0$ , we obtain the same result as in 2. \\
5.  If $|\xi_0| \sim |\xi_1| \gtrsim |\xi_2|$ and $s_2 \le 0$ we obtain
$$ I \lesssim \int_* \frac{\widehat{u}_1(\xi_1,\tau_1)}{\angles{\xi_1}^{s_0+s_1+s_2} \angles{|\xi_1|-|\tau_1|}^{\half+}}  \frac{\widehat{u}_2(\xi_2,\tau_2)}{\angles{|\xi_2|-|\tau_2|}^{\half+}}  \widehat{u}_0(\xi_0,\tau_0)  \lesssim \|u_1\|_{L^2_{xt}} \|u_2\|_{L^2_{xt}} \, , $$
because under our asumption $s_0+s_1+s_2 > \frac{n-1}{2}$ we obtain by Prop. \ref{Str}:
$$ \|uv\|_{L^2_{xt}} \le \|u\|_{L^2_t L^{\infty}_x} \|v\|_{L^{\infty}_x L^2_t} \lesssim \| u \|_{H^{s_0+s_1+s_2,\half+}} \|v\|_{H^{0,\half+}} \, . $$
6. If $|\xi_1| \ge |\xi_2|$ , $s_2 \le 0$ and $s_0 \le 0$ , or \\
7. If $|\xi_0| \sim |\xi_2| \ge |\xi_1|$ and $s_1 \le 0$ , or \\
8. If $|\xi_1| \le |\xi_2|$ , $s_1 \le 0$ and $s_0 \le 0$ , the same argument applies.

Thus we are done, if $s_0 \le 0$ , and also, if $s_0 \ge 0$ , and $|\xi_0| \sim |\xi_2 | \ge |\xi_1|$ or $ |\xi_0| \sim |\xi_1 | \ge |\xi_2|$ . 

It remains to consider the following case: $|\xi_0| \ll |\xi_1| \sim |\xi_2|$ and $s_0 > 0$ . We apply Prop. \ref{TheoremE} which gives
$$ \|uv\|_{H^{-s_0,0}} \lesssim \|u\|_{H^{s_1,\half+}} \|v\|_{H^{s_2,\half+}} \, , $$
under the conditions $0 < s_0 < \frac{n}{2} - 1 $ , $s_0+s_1+s_2 = \frac{n-1}{2} $ and $s_1,s_2 < \frac{n-1}{4}$ . The last condition is not necessary in our case $|\xi_1| \sim |\xi_2|$ . Remark that this implies $s_1+s_2 > \half$ , so that $s_0+s_1+s_2+s_1+s_2 > \frac{n}{2}$ .  , The second condition can now be replaced by $s_0+s_1+s_2 \ge \frac{n-1}{2}$ , because we consider inhomogeneous spaces.

Finally we consider the case $|\xi_0| \ll |\xi_1| \sim |\xi_2|$ and $s_0 \ge \frac{n}{2}-1 $ .
If $s_0 > \frac{n}{2} $ and $s_1+s_2 \ge 0$ we obtain the claimed estimate by Sobolev
$$ \|uv\|_{H^{-s_0,0}} \lesssim \|u\|_{H^{0,\half+}} \|v\|_{H^{0,\half+}} \le  \|u\|_{H^{s_1,\half+}} \|v\|_{H^{s_2,\half+}} \, . $$
We now interpolate the special case
$$ \|uv\|_{H^{-\frac{n}{2}-,0}} \lesssim \|u\|_{H^{0,\half+}} \|v\|_{H^{0,\half+}}  $$
with the following estimate
$$ \|uv\|_{H^{1-\frac{n}{2}+,0}} \lesssim \|u\|_{H^{\frac{1}{4}+,\half+}} \|v\|_{H^{\frac{1}{4}+,\half+}} \, , $$
which follows from Prop. \ref{TheoremE} . We obtain
$$\|uv\|_{H^{-s_0-,0}} \lesssim \|u\|_{H^{k+,\half+}} \|v\|_{H^{k+,\half+}} \, , $$
where $s_0 = (1-\theta)\frac{n}{2} - \theta(1-\frac{n}{2}) = \frac{n}{2} - \theta \, \Leftrightarrow \, \theta = \frac{n}{2}-s_0$ , $0 \le \theta \le 1 $ , $k=\frac{\theta}{4} = \frac{n}{8}-\frac{s_0}{4}.$  Using our asumption  $(s_0+s_1+s_2)+s_1+s_2 > \frac{n}{2} \, \Leftrightarrow \, \frac{n}{2}-s_0 < 2(s_1+s_2),$  we obtain $0 \le k < \frac{s_1+s_2}{2}$ .  Because $|\xi_1| \sim |\xi_2|$ , we obtain
$$ \|uv\|_{H^{-s_0,0}} \lesssim \|u\|_{H^{s_1,\half+}} \|v\|_{H^{s_2,\half+}} $$
for $(s_0+s_1+s_2)+s_1+s_2 > \frac{n}{2}$ and $s_1+s_2 \ge 0$ .
\end{proof}
\begin{Cor}
\label{Cor.3.1}
Under the assumptions of Prop. \ref{Prop.3.6}
$$ \|uv\|_{H^{-s_0,0}} \lesssim \|u\|_{H^{s_1,\half-}} \|v\|_{H^{s_2,\half-}}  \, . $$
\end{Cor}
\begin{proof}
This follows by bilinear interpolation of the estimate of Prop. \ref{Prop.3.6} with the estimate
$$\|uv\|_{H^{N,0}} \lesssim \|u\|_{H^{N,\frac{1}{4}+}} \|v\|_{H^{N,\frac{1}{4}+}}  \, , $$
where, say, $N > \frac{n}{2}$ , which follows by Sobolev apart from the special case $s_1=-s_2,$  in which we interpolate with the estimate
$$\|uv\|_{H^{-N,0}} \lesssim \|u\|_{H^{N,\frac{1}{4}+}} \|v\|_{H^{-N,\frac{1}{4}+}}   $$
in order to save the condition $s_1=-s_2$ .
\end{proof}
The following multiplication law is well-known:
\begin{prop} {\bf (Sobolev multiplication law)}
\label{SML}
Let $n\ge 2$ , $s_0,s_1,s_2 \in \mathbb{R}$ . Assume
$s_0+s_1+s_2 > \frac{n}{2}$ , $s_0+s_1 \ge 0$ ,  $s_0+s_2 \ge 0$ , $s_1+s_2 \ge 0$. Then the following product estimate holds:
$$ \|uv\|_{H^{-s_0}} \lesssim \|u\|_{H^{s_1}} \|v\|_{H^{s_2}} \, .$$
\end{prop}

\section{Proof of Theorem \ref{Theorem2}}
\begin{proof} 
It is by now standard that the claimed result follows by the contraction mapping principle in connection with the linear theory, if the following estimates hold:
\begin{align}
\label{3.6}
\| \Lambda_m^{-1} \mathcal{M}(\phi_+,\phi_-,A_+,A_-)\|_{X^{s,-\frac{1}{2}++}_{\pm}} & \lesssim R^2 + R^3 
\end{align}
and
\begin{align}
\label{3.7}
\| | D^{-1} \mathcal{N}(\phi_+,\phi_-,A_+,A_-)\|_{X^{r,-\epsilon_0 +}_{\pm}} & \lesssim R^2 + R^3 \quad \mbox{for} \, n=3 \\
\label{3.8}
\| D^{-1+\epsilon} \mathcal{N}(\phi_+,\phi_-,A_+,A_-)\|_{X^{r-\epsilon,-\epsilon_0 +}_{\pm}} & \lesssim R^2 + R^3 \quad \mbox{for} \, n=2 \, ,
\end{align}
where
$$
R  = \sum_{\pm} (\|\phi_{\pm}\|_{X^{s,\frac{1}{2}+}_{\pm}} + \| D A_{\pm}\|_{X^{r-1,1-\epsilon_0}_{\pm}}) \quad \mbox{for} \, n=3 \, , $$
$$
R  = \sum_{\pm} (\|\phi_{\pm}\|_{X^{s,\frac{1}{2}+}_{\pm}} + \| D^{1-\epsilon} A_{\pm}\|_{X^{r-1+\epsilon,1-\epsilon_0}_{\pm}}) \quad \mbox{for} \, n=2 \,  $$
and similar estimates for the differences, which follow by the same arguments in view of the multilinear character of the terms.

In the sequel we estimate the various nonlinear terms, where we use repeatedly the estimate $\|u\|_{H^{l,b}} \le \|u\|_{X^{l,b}_{\pm}}$ for $ b \ge 0$ and the reverse estimate for $b \le 0$. \\
{\bf Claim 1:} For $n \ge 3$ the following estimate holds: $$\|A^{\mu} \partial_{\mu} \phi \|_{H^{s-1,-\half+2\epsilon}} \lesssim \| DA\|_{H^{r-1,1-\epsilon_0}} \|\phi\|_{H^{s,\half+\epsilon}} $$
{\bf Proof:} By Tao \cite{T1}, Cor. 8.2 we may replace $DA$ by $\Lambda A$ . We apply (\ref{Q1}) and reduce to the following 7 estimates:
\begin{align}
\label{C1.1}
\|uv\|_{H^{s-1,0}} & \lesssim \|u\|_{H^{r+\half-2\epsilon,1-\epsilon_0}} \|v\|_{H^{s-1,\half+\epsilon}} \, , \\
\label{C1.2}
\|uv\|_{H^{s-1,0}} & \lesssim \|u\|_{H^{r,1-\epsilon_0}} \|v\|_{H^{s-\half-2\epsilon,\half+\epsilon}} \, , \\
\label{C1.3}
\|uv\|_{H^{s-1,-\half+2\epsilon}} & \lesssim \|u\|_{H^{r+\half,\half-\epsilon_0}} \|v\|_{H^{s-1,\half+\epsilon}} \, , \\
\label{C1.4}
\|uv\|_{H^{s-1,-\half+2\epsilon}} & \lesssim \|u\|_{H^{r,\half-\epsilon_0}} \|v\|_{H^{s-\half,\half+\epsilon}} \, , \\
\label{C1.5}
\|uv\|_{H^{s-1,-\half+2\epsilon}} & \lesssim \|u\|_{H^{r+\half,1-\epsilon_0}} \|v\|_{H^{s-1,0}} \, , \\
\label{C1.6}
\|uv\|_{H^{s-1,-\half+2\epsilon}} & \lesssim \|u\|_{H^{r,1-\epsilon_0}} \|v\|_{H^{s-\half,0}} \, , \\
\label{C1.7}
\|uv\|_{H^{s-1,-\half+\epsilon}} & \lesssim \|u\|_{H^{r,1-\epsilon_0}} \|v\|_{H^{s,\half+\epsilon}} \, .
\end{align}
(\ref{C1.7}) follows immediately from Prop. \ref{SML}, because $r+1 > \frac{n}{2}$ and $r \ge s-1$ . The other estimates follow from Prop. \ref{Prop.3.6} and Cor. \ref{Cor.3.1} by choosing the parameters as follows: \\
(\ref{C1.1}): $s_0=1-s$ , $s_1=r+\half-2\epsilon$ , $s_2=s-1$ , so that $s_0+s_1+s_2 = r+\half-2\epsilon > \frac{n-1}{2} $ for  $r > \frac{n}{2} -1$ , and $s_1+s_2=r+s-\half-2\epsilon > n - \frac{9}{4} \ge \frac{3}{4}$ for $r > \frac{n}{2}-1$ , $s > \frac{n}{2} - \frac{3}{4}$ . \\
(\ref{C1.2}): similarly \\
(\ref{C1.3}): $s_0=1-s$ , $s_1=r+\half$ , $s_2=s-1$ , so that $s_1+s_2=r+s-\half > \frac{3}{4}$ . Here we used that we allow  $DA \in H^{r-1,1-\epsilon_0}$     instead of $DA \in H^{r-1,\half +}$ . \\
(\ref{C1.4}): $s_0 = r$ , $s_1=s-\half$ , $s_2 = 1-s$ , so that $s_1+s_2 = \half$ . \\
(\ref{C1.5}): $s_0= s-1$ , $s_1= r+\half$ , $s_2=1-s$ , so that $s_0+s_1+s_2+s_1+s_2= 2r-s+2 > \frac{n}{2}$ by our assumption $2r-s > \frac{n-3}{2}$ . \\
(\ref{C1.6}): $s_0 = s-\half$ , $s_1=r$ , $s_2=1-s$ , so that $s_0+s_1+s_2+s_1+s_2 = 2r-s+\frac{3}{2} > \frac{n}{2}$. \\[0.5em]
{\bf Claim 2:} For $n=2$ the following estimate holds:
$$ \|A^{\mu} \partial_{\mu} \phi \|_{H^{s-1,-\half+2\epsilon}} \lesssim \|D^{1-\epsilon_1}A\|_{H^{r-1+\epsilon_1,\frac{3}{4}+\epsilon}} \|\phi\|_{H^{s,\half+\epsilon}} $$
{\bf Proof:} We may replace $D^{1-\epsilon_1} A$ by $\Lambda^{1-\epsilon_1} A$ by \cite{T1}, Cor. 8.2 and argue similarly as for claim 1 using Theorem \ref{Theorem3}. By (\ref{Q1}) we reduce to the same estimates, where in (\ref{C1.1}),(\ref{C1.2})(\ref{C1.5}),(\ref{C1.6}) and (\ref{C1.7}) we replace the $H^{l,1-\epsilon_0}$-norms by $H^{l,\frac{3}{4}+\epsilon}$-norms, which makes no essential difference. According to Theorem \ref{Theorem3} we have to show $s_0+s_1+s_2 > \frac{3}{4}$ . Using that in all these cases $s_0+s_1+s_2 = r+\half$ , this is fulfilled for $r > \frac{1}{4}$ . Moreover we need $s_0+s_1+s_2+s_1+s_2 > 1$ . In (\ref{C1.1}) and (\ref{C1.2}) we have $s_1+s_2 = r+s-\half > \frac{1}{4}$ by our assumptions $r>\frac{1}{4}$ , $s>\half$ , so that this condition is satisfied, whereas in (\ref{C1.5}) we have $s_0+s_1+s_2 = 2r-s+2 > 1$ by our assumption $2r-s > - \half$ , and in (\ref{C1.6}) we obtain $s_0+s_1+s_2+s_1+s_2 = 2r-s+\frac{3}{2} > 1$ as well. In (\ref{C1.4}) we replace $H^{r,\half-\epsilon_0}$ by $H^{r,\frac{1}{4}+\epsilon}$ , which makes no essential difference, because $s_1+s_2 = \half$ . In (\ref{C1.3}) we replace $H^{r,\half-\epsilon_0}$ by $H^{r,\frac{1}{4}+\epsilon}$ and use Theorem \ref{Theorem3} with parameters $s_0=r+\half$ , $s_1=s-1$ , $s_2=1-s$ , $b_0=\frac{1}{4}+\epsilon$ , $b_1= \half + \epsilon$ , $b_2= \half -2\epsilon$ , so that $s_0+s_1+s_2 = r+\half > \frac{3}{4}$ and $s_0+s_1+s_2 + s_1 + s_2 + b_0 = r+\frac{3}{4}+\epsilon > 1$ , where it is essential to allow $D^{1-\epsilon_1} A \in H^{r-1+\epsilon_1,\frac{3}{4}+\epsilon}$ instead of $D^{1-\epsilon_1} A \in H^{r-1+\epsilon_1,\half+\epsilon}$ . Finally, (\ref{C1.7}) follows by Prop. \ref{SML}. \\[0.5em]
{\bf Claim 3:} If $ n \ge 3$ we obtain
$$\| D^{-1}(\phi \partial \phi)\|_{H^{r,0}} \lesssim \|\phi\|^2_{H^{s,\half+\epsilon}} $$
{\bf Proof:} We may replace $D^{-1}$ by $\Lambda^{-1}$ and use Prop. \ref{Prop.3.6} with $s_0=1-r$ , $s_1 = s$,  $s_2= s-1$ , so that $s_0+s_1+s_2= 2s-r >\frac{n-1}{2}$ by assumption, and also $s_1+s_2= 2s-1 > \half$ , if $s > \frac{3}{4}$ . We also need the assumption $s \ge r-1$ . \\[0.5em]
{\bf Claim 3':} For $ n =2 $ we obtain
$$ \| D^{-1+\epsilon_1}(\phi \partial \phi)\|_{H^{r-\epsilon_1,-\frac{1}{4}+2\epsilon}} \lesssim \|\phi\|^2_{H^{s,\half+\epsilon}} $$
{\bf Proof:} We may replace $D^{-1+\epsilon_1}$ by $\Lambda^{-1+\epsilon_1}$ . We use Theorem \ref{Theorem3} with $s_0,s_1,s_2$ as in claim 3 , but now $b_0 = \frac{1}{4}-2\epsilon$ , $b_1=b_2=\half + \epsilon$ , so that we need our assumption $s_0+s_1+s_2 = 2s-r > \frac{3}{4}$ and $(s_0+s_1+s_2)+(s_1+s_2+b_0) = (2s-r)+(2s-1) + \frac{1}{4}-2\epsilon > 1$ for $s \ge \half$ . \\

It remains to consider the cubic nonlinearities.\\
{\bf Claim 4:} In the case $n \ge 3$ the following estimate holds:
$$ \|A_{\mu} A_{\nu} \phi \|_{H^{s-1,-\half+2\epsilon}} \lesssim \|DA_{\mu}\|_{H^{r-1,\half+\epsilon}} \|DA_{\nu}\|_{H^{r-1,\half+\epsilon}} \|\phi\|_{H^{s,\half+\epsilon}} $$
{\bf Proof:} 1. Assume that at least one of the $A$-factors has frequencies $\ge 1$. We may replace in this case $DA_{\mu}$ by $\Lambda A_{\mu}$ and $DA_{\nu}$ by $\Lambda A_{\nu}$ by \cite{T1}, Cor. 8.2. \\
a. If $s > \frac{n}{2} - \half$ we apply Prop. \ref{Prop.3.6} twice to obtain
\begin{align*}
\|A_{\mu} A_{\nu} \phi \|_{H^{s-1,-\half+2\epsilon}} & \lesssim \|A_{\mu} A_{\nu}\|_{H^{s-1,0}} \|\phi\|_{H^{s,\half+\epsilon}} \\
& \lesssim \|A_{\mu}\|_{H^{r,\half+\epsilon}} \| A_{\nu}\|_{H^{r,\half+\epsilon}} \|\phi\|_{H^{s,\half+\epsilon}} \, ,
\end{align*}
where we choose for the first step $s_0=s-1$ , $s_1 = 1-s$ , $s_2=s$ , so that $s_0+s_1+s_2= s > \frac{n}{2}-\half$ by assumption and $s_1+s_2=1$ . For the second step the choice $s_0=1-s$ , $s_1=s_2=r$ gives $s_0+s_1+s_2 = 2r-s+1 > \frac{n-1}{2}$ by our assumption and also $s_1+s_2 = 2r >\half$ . \\
b. If $s \le \frac{n}{2}-\half$ we obtain similarly
\begin{align*}
\|A_{\mu} A_{\nu} \phi \|_{H^{s-1,-\half+2\epsilon}} & \lesssim \|A_{\mu} A_{\nu}\|_{H^{\frac{n}{2}-\frac{3}{2}+,0}} \|\phi\|_{H^{s,\half+\epsilon}} \\
& \lesssim \|A_{\mu}\|_{H^{r,\half+\epsilon}} \| A_{\nu}\|_{H^{r,\half+\epsilon}} \|\phi\|_{H^{s,\half+\epsilon}} \, .
\end{align*}
For the first step we choose $s_0=\frac{n}{2}-\frac{3}{2}+$ , $s_1= s$ , $s_2=1-s$ , so that $s_0+s_1+s_2 = \frac{n}{2}-\half + $ and $s_1+s_2=1$ , whereas for the second step $s_0=\frac{3}{2}-\frac{n}{2}-$ , $s_1=s_2=r$ , so that $s_0+s_1+s_2 = 2r + \frac{3}{2}-\frac{n}{2}- > \frac{n}{2}-\half $ for $r > \frac{n}{2}-1$ , and $s_1+s_2 = 2r > 1$ . \\
2. If both factors $A_{\mu}$ and $A_{\nu}$ have frequencies $\le 1$ the frequencies of the product and $\phi$ are equivalent, so that we may crudely estimate as follows:
\begin{align*}
&\|A_{\mu} A_{\nu} \phi \|_{H^{s-1,-\half+2\epsilon}}  \lesssim \|A_{\mu} A_{\nu} \Lambda^{s-1} \phi \|_{L^2_x L^2_t} \lesssim \|A_{\mu}\|_{L^{\infty}_t L^{\infty}_x} \|A_{\nu}\|_{L^{\infty}_t L^{\infty}_x} \|\phi\|_{L^2_t H^s_x} \\
& \lesssim \|DA_{\mu}\|_{L^{\infty}_t L^2_x} \|DA_{\nu}\|_{L^{\infty}_t L^2_x} \|\phi\|_{L^2_t H^s_x} \lesssim  \|DA_{\mu}\|_{H^{r-1,\half+\epsilon}} \|DA_{\nu}\|_{H^{r-1,\half+\epsilon}} \|\phi\|_{H^{s,\half+\epsilon}} \, .
\end{align*}
{\bf Claim 4':} In the case $n=2$ the following estimate holds:
$$ \|A_{\mu} A_{\nu} \phi \|_{H^{s-1,-\half+2\epsilon}} \lesssim \|D^{1-\epsilon_1} A_{\mu}\|_{H^{r-1+\epsilon_1,\half+\epsilon}} \|D^{1-\epsilon_1} A_{\nu}\|_{H^{r-1+\epsilon_1,\half+\epsilon}} \|\phi\|_{H^{s,\half+\epsilon}} $$
{\bf Proof:} We argue similarly as for claim 4. \\
1. If at least one $A$-factor has frequencies $\ge 1$ we use \cite{T1}, Cor.8.2 again and my replace $D^{1-\epsilon_1} A_{\mu}$ and  $D^{1-\epsilon_1} A_{\nu}$ by $\Lambda^{1-\epsilon_1} A_{\mu}$ and  $\Lambda^{1-\epsilon_1} A_{\nu}$ , respectively , and use Theorem \ref{Theorem3} twice. 
We obtain
\begin{align*}
\|A_{\mu} A_{\nu} \phi \|_{H^{s-1,-\half+2\epsilon}} & \lesssim \|A_{\mu} A_{\nu}\|_{H^{s-1+,0}} \|\phi\|_{H^{s,\half+\epsilon}} \\
& \lesssim \|A_{\mu}\|_{H^{r,\half+\epsilon}} \| A_{\nu}\|_{H^{r,\half+\epsilon}} \|\phi\|_{H^{s,\half+\epsilon}} 
\end{align*}
choosing $s_0=s-1+$ , $s_1=s$ , $s_2=1-s$ , so that $s_0+s_1+s_2=s+ > \frac{3}{4}$ and $s_1+s_2 = 1$ , for the first estimate, and choosing $s_0=1-s-$ , $s_1=s_2=r$ for the
 second estimate, so that $s_0+s_1+s_2 = 2r-s+1- > \frac{3}{4}$, where we used our assumption $2r-s > -\frac{1}{4}$ , and $s_1+s_2 = 2r > \half$ . \\
2. If both $A$-factors have frequencies $\le 1$ , we may argue as in the proof of claim 4. \\[0.5em]
{\bf Claim 5:} In the case $ n\ge 3$ we obtain
$$ \|D^{-1} (A \phi \psi)\|_{H^{r,0}} \lesssim \|DA\|_{H^{s-1,\half + \epsilon}} \|\phi\|_{H^{s,\half+\epsilon}} \|\psi\|_{H^{s,\half+\epsilon}} $$
and in the case $n=2$ the following estimate holds:
$$ \|D^{-1+\epsilon_1} (A \phi \psi)\|_{H^{r-\epsilon_1,0}} \lesssim \|D^{1-\epsilon_1}A\|_{H^{s-1+\epsilon_1,\half + \epsilon}} \|\phi\|_{H^{s,\half+\epsilon}} \|\psi\|_{H^{s,\half+\epsilon}} \, .$$
{\bf Proof:} 1. If the frequencies of the product or $A$ are  $\ge 1$ , we use \cite{T1}, Cor. 8.2 to replace everywhere $D$ by $\Lambda$ . \\
a. If $r \le \frac{n}{2}$ , we use Prop. \ref{SML} and Prop. \ref{Prop.3.6} and obtain
$$\|A \phi \psi\|_{H^{r-1,0}} \lesssim \|A\|_{H^{r,\half+\epsilon}} \|\phi \psi\|_{H^{\frac{n}{2}-1+,0}} \lesssim \|A\|_{H^{r,\half+\epsilon}} \|\phi\|_{H^{s,\half+\epsilon}} \|\psi\|_{H^{s,\half+\epsilon}} \, , $$
where for the first step $s_0=1-r$ , $s_1=r$ , $s_2= \frac{n}{2}-1+$ , so that $s_0+s_1+s_2 = \frac{n}{2}+$,  and for the second step $s_0=1-\frac{n}{2}-$ , $s_1=s_2=s$ , so that $s_0+s_1+s_2 = 1-\frac{n}{2}+2s- > \frac{n}{2}-\half$ by our assumption $s > \frac{n}{2}-\frac{3}{4}$ for $n\ge 3$ and $s_0+s_1+s_2 > 1$  by  the assumption $s > \half$ for $n=2$, and $s_1+s_2 = 2s > \half$ . \\
2. If $ r > \frac{n}{2}$ , we obtain by Prop. \ref{SML} (or Prop. \ref{Prop.3.6}) :
$$\|A \phi \psi\|_{H^{r-1,0}} \lesssim \|A\|_{H^{r,\half+\epsilon}} \|\phi \psi\|_{H^{r-1+,0}} \lesssim \|A\|_{H^{r,\half+\epsilon}} \|\phi\|_{H^{s,\half+\epsilon}} \|\psi\|_{H^{s,\half+\epsilon}} \, . $$
The first step holds, because $r > \frac{n}{2}$ , and for the second step we choose $s_0= 1-r-$,  $s_1=s_2=s$ , so that $s_0+s_1+s_2 = 2s-r+1- > \frac{n}{2}-\half$ , because by assumption $2s-r > \frac{n-1}{2}$ for $n\ge 3$ and $2s-r > \frac{3}{4}$ for $n=2$ . We also use $s \ge r-1$ . \\
2. If the frequencies of the product and $A$ are $\le 1$ , we obtain  in the case $n \ge 3$ :
\begin{align*}
\|D^{-1} (A\phi \psi)\|_{H^{r,0}} &\lesssim \|D^{-1}(A\phi \psi)\|_{H^{-2,0}} \lesssim \|A\phi\|_{H^{-2,0}} \|\psi\|_{H^{s,\half+\epsilon}} \\
& \lesssim \|DA\|_{H^{r-1,\half+\epsilon}} \|\phi\|_{H^{s,\half+\epsilon}} \|\psi\|_{H^{s,\half+\epsilon}} \, ,
\end{align*}
where we replaced $D^{-1}$ by $\Lambda^{-1}$ and $D$ by $\Lambda$ by \cite{T1}, Cor. 8.2. , and used Prop. \ref{Prop.3.6} for the second inequality  with $s_0=3$ , $s_1=-2$ , $s_2 =s$ , so that $s_0+s_1+s_2 = s+1 > \frac{n}{2}$ , and for the last inequality with $s_0=2$ , $s_1= r$ , $s_2=s$ , so that $s_0+s_1+s_2 =  2+r+s > n $ . In the case $n=2$ we may argue similarly replacing $D^{-1+\epsilon_1}$ by $\Lambda^{-1+\epsilon_1}$ and $D^{1-\epsilon_1}$ by $\Lambda^{1-\epsilon_1}$ by \cite{T1}, Cor. 8.2. 
\end{proof}

\section{Proof of Theorem \ref{Theorem1}}
We start with the solution $(\phi_{\pm},A_{\pm})$ of (\ref{3.1}),(\ref{3.2}) given by Theorem \ref{Theorem2}. Defining $\phi:=\phi_+ + \phi_-$ , $A:=A_+ + A_-$ we immediately see that $\partial_t \phi = i \Lambda_m (\phi_+ - \phi_-)$ , $\partial_t A = i D(A_+ - A_-)$, so that $\mathcal{N}$ in (\ref{3.2}) is the same as $N$ in (\ref{2.2}) and (\ref{16}). 

Moreover  by (\ref{3.2}):
\begin{align*}
\square A & = (i\partial_t - D)(i\partial_t + D)A_+ + (i\partial_t + D)(i\partial_t - D)A_- \\
& = -(i \partial_t - D)(2D)^{-1} \mathcal{N}(A,\phi) + (i\partial_t + D) (2D)^{-1} \mathcal{N}(A,\phi) = N(A,\phi) \, .
\end{align*}
Thus $A$ satisfies (\ref{16}). We also have $\phi(0) = \phi_0$ , $\partial_t \phi = \phi_1$ , $A(0) = a_0$ , $\partial_t A(0) = \dot{a}_0$.

Next we prove that the Lorenz condition $\partial^{\mu} A_{\mu} =0$ is satisfied. We define
$$
u  := \partial^{\mu} A_{\mu} = -\partial_t A_0 + \partial^j A_j \quad , \quad
u_{\pm}  := -\partial_t A_{0_{\pm}} + \partial^j A_{j_{\pm}} \, . $$
By (\ref{2.2}) we obtain
\begin{align*}
&(i\partial_t \pm D)u_{\pm} = -\partial_t(i\partial_t \pm D)A_{0_{\pm}} + \partial^j(i\partial_t \pm D)A_{j_{\pm}} \\
&= (\pm 2D)^{-1}(\partial_t(Im(-\phi \overline{i \Lambda_m(\phi_+ - \phi_-)}) -A_0|\phi|^2)-\partial^j(Im(-\phi\overline{\partial_j \phi}) - A_j |\phi|^2)) \\
&= (\pm 2 D)^{-1} (Im(\phi \Lambda_m(-\overline{i\partial_t \phi_+} + \overline{i\partial_t \phi_-})) + Im(\phi \overline{\Delta \phi}) + \partial^{\mu}(A_{\mu} |\phi|^2)) \, .
\end{align*}
Now we have
$$ Im(\phi \overline{\Delta \phi})=Im(\phi(m^2\overline{\phi} - \Lambda_m \Lambda_m \overline{\phi})) = -Im(\phi \Lambda_m(\Lambda_m \overline{\phi}_+ + \Lambda_m \overline{\phi}_-)) \, , $$
so that by (\ref{3.1}) we obtain
\begin{align*}
&(i\partial_t \pm D)u_{\pm} \\
&= (\pm 2 D)^{-1} (Im(\phi \Lambda_m(-(\overline{i\partial_t + \Lambda_m) \phi_+} + (\overline{i\partial_t - \Lambda_m)\phi_-})) +  \partial^{\mu}(A_{\mu} |\phi|^2)) \\
& = (\pm 2D)^{-1}(Im(\phi \overline{\mathcal{M}(\phi_+,\phi_-,A_+,A_-)}) + A_{\mu} \partial^{\mu}(|\phi|^2) + |\phi|^2 u) \\
& =: (\pm 2 |\nabla |)^{-1} R(A,\phi) \, .
\end{align*}
By (\ref{3.3}) and the second equation in (\ref{2.2a})
\begin{align*}
\mathcal{M}(\phi_+,\phi_-,A_+,A_-) & = 2 \sum_{\pm_1,\pm_2} \pm_2 \mathcal{A}_{(\pm_1,\pm_2)} (A_{0_{\pm_1}},\phi_{\pm_2}) + 2iP_2 + A_{\mu} A^{\mu} \phi \\
&=2((A_{0_+} + A_{0_-}) \Lambda(\phi_+ - \phi_-) \\
& \hspace{2em} + D^{-1} \nabla(A_{0_+} - A_{0_-}) \cdot \nabla (\phi_+ + \phi_-)) +2iP_2 + A_{\mu} A^{\mu} \phi \\
& = 2i(-A_0 \partial_t \phi -  (-\Delta)^{-1} \nabla \partial_t A_0 \cdot \nabla \phi) + 2iP_2 + A_{\mu}A^{\mu} \phi \, .
\end{align*}
Now by the definition of $P_2$
\begin{align*}
 P_2=R^k(R_j A_k - R_k A_j) \partial^j \phi &= (-\Delta)^{-1} \partial^k(\partial_j A_k - \partial_k A_j) \partial^j \phi \\
 &= (-\Delta)^{-1} \nabla(\partial^k A_k)\cdot \nabla \phi + A_j \partial^j \phi 
 \end{align*}
 and by the definition of $u$
$$ (-\Delta)^{-1} \nabla \partial_t A_0 \cdot \nabla \phi = (-\Delta)^{-1} \nabla(\partial^j A_j -u)\cdot \nabla \phi \, , $$
so that we obtain
\begin{align*}
\mathcal{M}(\phi_+,\phi_-,A_+,A_-) & = 2i(-A_0 \partial_t \phi + (-\Delta)^{-1} \nabla u \cdot \nabla \phi + A_j \partial^j \phi) +  A_{\mu} A^{\mu} \phi\\
& = 2i(A_{\mu} \partial^{\mu} \phi + (-\Delta)^{-1} \nabla u \cdot \nabla \phi) + A_{\mu} A^{\mu} \phi \, ,
\end{align*}
which implies
\begin{align*}
&R(A,\phi) \\
& = Im(-\phi 2i(A_{\mu} \partial^{\mu} \overline{\phi} + (-\Delta)^{-1} \nabla u \cdot \nabla \overline{\phi}) + A_{\mu} \partial^{\mu}(|\phi|^2) +|\phi|^2 u + Im(\phi A_{\mu} A^{\mu} \overline{\phi}) \\
& = 2 Re(-\phi A_{\mu} \partial^{\mu} \overline{\phi}) -2 Re(\phi (-\Delta)^{-1} \nabla u \cdot \nabla \overline{\phi}) + 2 Re (A_{\mu} \phi \partial^{\mu} \overline{\phi}) + |\phi|^2 u \\
& = -2 Re(\phi \nabla \overline{\phi}) \cdot (-\Delta)^{-1} \nabla u + |\phi|^2 u \, ,
\end{align*}
so that
$$ (i\partial_t \pm D)u_{\pm} = (\pm 2D)^{-1}(-2 Re(\phi \nabla \overline{\phi}) \cdot (-\Delta)^{-1} \nabla u + |\phi|^2 u) \, , $$
and thus $u$ fulfills the linear equation
$$ \square u = -2 Re(\phi \nabla \overline{\phi}) \cdot (-\Delta)^{-1} \nabla u + |\phi|^2 u \, . $$
The data of $u$ fulfill by (\ref{10}) and (\ref{12}):
$$ u(0) = -\partial_t A_0(0) + \partial^j A_j(0) = - \dot{a}_{00} + \partial^j a_{0j} = 0 $$
and, using that $A$ is a solution of (\ref{16}) and also (\ref{15}):
\begin{align*}
\partial_t u(0) & = -\partial_t^2 A_0(0) + \partial_t \partial^j A_j(0) = -\partial^j \partial_j A_0(0) - j_{0_{|t=0}} + \partial_t \partial^j A_j(0) \\
& = -\partial^j (\partial_j A_0(0) - \partial_t A_j(0)) - j_{0_{|t=0}} = \partial^j F_{0j}(0) -  j_{0_{|t=0}} \\
& = Im(\phi_0 \overline{\phi}_1) - j_{0_{|t=0}} =  j_{0_{|t=0}}- j_{0_{|t=0}} = 0 \, .
\end{align*}
By uniqueness this implies $u = 0$. Thus the Lorenz condition $\partial^{\mu} A_{\mu} =0$ is satisfied. Under the Lorenz condition however we know that $\mathcal{M}$ in (\ref{3.1}) is the same as $M$ in (\ref{2.1}) and (\ref{17}).  Moreover by (\ref{3.1}) we obtain
\begin{align*}
(\square - m^2) \phi & = (i\partial_t - \Lambda_m)(i \partial_t + \Lambda_m) \phi_+ + (i\partial_t + \Lambda_m)(i \partial_t - \Lambda_m) \phi_- \\
& = (-(i \partial_t - \Lambda_m)  + (i \partial_t + \Lambda_m))(2 \Lambda_m)^{-1} \mathcal{M}(\phi_+,\phi_-,A_+,A_-) \\
& = \mathcal{M}(\phi_+,\phi_-,A_+,A_-) = M(\phi,A) \, .
\end{align*}
Thus $\phi$ satisfies (\ref{17}). Because (\ref{16}),(\ref{17}) is equivalent to (\ref{1}),(\ref{2}), where $F_{\mu \nu} := \partial_{\mu} A_{\nu} - \partial_{\nu} A_{\mu}$, we also have that $F_{k0}$ satisfies (\ref{2.10}) and $F_{kl}$ satisfies (\ref{2.11}).\\[0.5em]

What remains to be shown are the following properties of the electromagnetic field $F_{\mu \nu}$:
\begin{align*}
&F_{\mu \nu} \in X_+^{s-1,\frac{1}{2}+}[0,T] + X_-^{s-1,\frac{1}{2}+}[0,T] \\
& \Longrightarrow \, F_{\mu \nu} \in H^{s-1,\half+}[0,T] \, , \,  \partial_t F_{\mu \nu} \in H^{s-2,\half+}[0,T]
& \mbox{in the case} \, n=3\, , \\
&D^{-\epsilon} F_{\mu \nu} \in H^{s-1+\epsilon,\frac{1}{2}+}[0,T] \, , \, \partial_t D^{-\epsilon} F_{\mu \nu} \in H^{s-2+\epsilon,\frac{1}{2}+}[0,T]
&\mbox{in the case} \, n=2 \, .
\end{align*}
By transformation of (\ref{2.10}) and (\ref{2.11}) into a first order system as before and
using well-known results for Bourgain type spaces these properties  in the case $n=3$ are reduced to:
\begin{align}
\label{4.10}
&F_{\mu \nu} (0) \in H^{s-1} \\
\label{4.11}
&\partial_t  F_{\mu \nu}(0) \in H^{s-2} \quad \mbox{and} \\
\label{4.12}
&D^{-1}\square F_{\mu \nu} \in X_+^{s-1,-\frac{1}{2}+}[0,T] + X_-^{s-1,-\frac{1}{2}+}[0,T] \, .
\end{align}
In the case $n=2$ we refer to \cite{KS} and have to show
\begin{align}
\label{4.13}
&D^{-\epsilon} F_{\mu \nu} (0) \in H^{s-1+\epsilon} \, ,
\\
\label{4.14}
&D^{-\epsilon} \partial_t  F_{\mu \nu}(0) \in H^{s-2+\epsilon} \quad \mbox{and} \\
\label{4.15}
&\Lambda_+^{-1} D^{-\epsilon} \square F_{\mu \nu} \in H^{s-1+\epsilon,-\frac{1}{2}+}[0,T]  \, ,
\end{align}
where in all cases $\square F_{\mu \nu} $ is given by (\ref{2.10}) and (\ref{2.11}). These properties for $n=2$ in fact imply $D^{-\epsilon} F_{\mu \nu} \in H^{s-1+\epsilon,\half+}[0,T]$ and $ \partial_t D^{-\epsilon} F_{\mu \nu} \in H^{s-2+\epsilon,\half+}[0,T]$
by \cite{KS}, Thm. 5.5, Prop. 5.5 and Prop. 5.6 or \cite{Se}, Thm. 1. \\
We use this approach in the two-dimensional case in order to avoid the unpleasant singularity of $D^{-1}$ . \\[0.5em]
We start to prove (\ref{4.12}) and (\ref{4.15}):
We first estimate the quadratic terms by (\ref{Q2}) and (\ref{Q3}). \\
{\bf Claim 1:} In the case $n \ge 3$ the following estimates hold:
\begin{align}
\label{Cl1.1a}
\|D^{-1} \bar{D}_{\pm}^{\half-2\epsilon} (D^{\half+2\epsilon} \phi_{\pm_1} D \phi_{\pm_2}) \|_{X^{s-1,-\half+2\epsilon}_{\pm}} \lesssim \|\phi_{\pm_1} \|_{X^{s,\half+\epsilon}_{\pm_1}} \|\phi_{\pm_2} \|_{X^{s,\half+\epsilon}_{\pm_2}} \\
\nonumber
\mbox{if} \, \pm_1 = \pm_2 \, , \\
\label{Cl1.1b}
\|D^{-1} D^{\half} \bar{D}_{\pm}^{\half-2\epsilon} (D^{\half+\epsilon} \phi_{\pm_1} D^{\half+\epsilon} \phi_{\pm_2}) \|_{X^{s-1,-\half+2\epsilon}_{\pm}} \lesssim \|\phi_{\pm_1} \|_{X^{s,\half+\epsilon}_{\pm_1}} \|\phi_{\pm_2} \|_{X^{s,\half+\epsilon}_{\pm_2}} \\
\nonumber
\mbox{if} \, \pm_1 \neq \pm_2 \, , \\
\label{Cl1.2}
\|D^{-1}( D^{\half}\bar{D}_{\pm_1}^{\half}  \phi_{\pm_1} D \phi_{\pm_2}) \|_{X^{s-1,-\half+2\epsilon}_{\pm}} \lesssim \|\phi_{\pm_1} \|_{X^{s,\half+\epsilon}_{\pm_1}} \|\phi_{\pm_2} \|_{X^{s,\half+\epsilon}_{\pm_2}} \\
\label{Cl1.3}
\|D^{-1} ( D^{\half}\phi_{\pm_1} D \bar{D}^{\half}_{\pm_2} \phi_{\pm_2}) \|_{X^{s-1,-\half+2\epsilon}_{\pm}} \lesssim \|\phi_{\pm_1} \|_{X^{s,\half+\epsilon}_{\pm_1}} \|\phi_{\pm_2} \|_{X^{s,\half+\epsilon}_{\pm_2}} \\
\label{Cl1.4}
\|D^{-1} ( \phi_{\pm_1} D \phi_{\pm_2}) \|_{X^{s-1,-\half+2\epsilon}_{\pm}} \lesssim \|\phi_{\pm_1} \|_{X^{s,\half+\epsilon}_{\pm_1}} \|\phi_{\pm_2} \|_{X^{s,\half+\epsilon}_{\pm_2}} \, .
\end{align}
{\bf Proof:} By \cite{T1}, Cor. 8.2 we may replace  $D^{-1}$ by $\Lambda^{-1}$ . \\
(\ref{Cl1.1a}): This reduces to
$$ \|uv\|_{X^{s-2,0}_{\pm}} \lesssim \|u\|_{X^{s-\half-2\epsilon,\half+\epsilon}_{\pm_1}} \|v\|_{X^{s-1,\half+\epsilon}_{\pm_2}} \, . $$
We use Prop. \ref{Prop.3.6} with $s_0=2-s$ , $s_1 = s-\half-2\epsilon$ , $s_2=s-1$ , so that $s_0+s_1+s_2 =s+\half-2\epsilon > \frac{n}{2}-\frac{1}{4}$ on our assumption $s > \frac{n}{2}-\frac{3}{4}$ . The condition $s_0+s_1+s_2 + s_1+s_2 > \frac{n}{2}$ is not needed  (and violated for $s$ close to $\frac{3}{4}$) in the case $\pm_1 = \pm_2$ for $n=3$ according to Thm. \ref{Theorem3}. If $n \ge 4$ this condition is in fact fulfilled, because $s_1+s_2=2s-\frac{3}{2}-2\epsilon > n-3 \ge 1$ assuming $s > \frac{n}{2}-\frac{3}{4}$ . \\
(\ref{Cl1.1b}): This reduces to
$$ \|uv\|_{H^{s-\frac{3}{2},0}} \lesssim \|u\|_{H^{s-\half-\epsilon,\half+\epsilon}} \|v\|_{H^{s-\half-\epsilon,\half+\epsilon}} \, ,$$
which holds by Prop. \ref{Prop.3.6} with $s_0=\frac{3}{2}-s$ , $s_1=s_2= s-\half-\epsilon$ , so that $s_0+s_1+s_2 = s +\half-2\epsilon > \frac{n}{2}-\frac{1}{4}$ and $s_1+s_2 = 2s-1-2\epsilon > n - \frac{5}{2} \ge \half $ for $n \ge 3$ and $s > \frac{n}{2}-\frac{3}{4}$ . \\
(\ref{Cl1.2}): We have to show
$$\|uv\|_{H^{s-2,-\half+2\epsilon}} \lesssim \|u\|_{H^{s-\half,0}} \|v\|_{H^{s-1,\half + \epsilon}} \, . $$
We use Prop. \ref{Prop.3.6} with $s_0=s-\half$ , $s_1=s-1$ , $s_2=2-s$ , so that $s_0+s_1+s_2=s+\half > \frac{n}{2}-\frac{1}{4}$ and $s_1+s_2=1$ . \\
(\ref{Cl1.3}): It suffices to prove
$$ \|uv\|_{H^{s-1,-\half + 2\epsilon}} \lesssim \|u\|_{H^{s-\half,\half + \epsilon}} \|v\|_{H^{s-1,0}} \, , $$
which is handled similarly. \\
(\ref{Cl1.4}): This follows easily by Prop. \ref{SML} for $s > \frac{n}{2}-1$ . \\[0.5em]

In the case $n=2$ we use the original system (\ref{2.10}),(\ref{2.11}) and replace the estimates (\ref{Q2}) and (\ref{Q3}) using the following result for the null forms $Q_{0i}(u,v)$ and $Q_{ij}(u,v)$ , which was given in \cite{KMBT}, Prop. 1.
\begin{align*}
 Q_{0i}(u,v) + Q_{ij}(u,v) & \lesssim D_+^{\half} D_-^{\half} (D_+^{\half} u D_+^{\half} v) + D_+^{\half}(D_+^{\half} D_-^{\half} u D_+^{\half} v) + D_+^{\half}(D_+^{\half} u D_+^{\half} D_-^{\half} v) 
\end{align*}
Interpolation with the trivial bound $D_+ u D_+ v$ gives (for $0 \le \epsilon \le \frac{1}{4}$) :
\begin{align*}
Q_{0i}(u,v) &+ Q_{ij}(u,v)  \lesssim D_+^{\half-2\epsilon} D_-^{\half-2\epsilon} (D_+^{\half+2\epsilon} u D_+^{\half+2\epsilon} v) \\
&+ D_+^{\half-2\epsilon}(D_+^{\half+2\epsilon} D_-^{\half-2\epsilon} u D_+^{\half+2\epsilon} v) + D_+^{\half-2\epsilon}(D_+^{\half+2\epsilon} u D_+^{\half+2\epsilon} D_-^{\half-2\epsilon} v) \, ,
\end{align*}
and by the fractional Leibniz rule for $D_+$  we obtain
\begin{align}
\label{Q4}
Q_{0i}(u,v) + Q_{ij}(u,v) &  \lesssim D_-^{\half-2\epsilon} (D_+^{\half+2\epsilon} u D_+ v)
 +D_+^{\half+2\epsilon} D_-^{\half-2\epsilon} u D_+ v                          \\ \nonumber
&+ D_+^{\half+2\epsilon} u D_-^{\half-2\epsilon}  D_+ v
+ D_-^{\half-2\epsilon}(D_+ u D_+^{\half+2\epsilon} v) \\ \nonumber
 &+ D_+D_-^{\half-2\epsilon} u D_+^{\half+2\epsilon}v
+ D_+ u D_+^{\half+2\epsilon} D_-^{\half-2\epsilon}v
 \, .
\end{align}
One easily checks that $\phi_{\pm} \in X^{s,\half+}_{\pm} $ implies $\Lambda_+ \phi \in H^{s-1,\half+}$ , and $D^{1-\epsilon_1} A_{\pm} \in X^{r-1+\epsilon_1,\half+}_{\pm}$ implies $ \Lambda_+ D^{1-\epsilon_1} A \in H^{r-2+\epsilon_1,\half+}$ . In the case $n=2$ we now use (\ref{Q4}) and may reduce to  the estimates for the nonlinearities in claims 2,4 and 5 below. \\[0.5em]
{\bf Claim 2:} For $n=2$ the following estimates hold:
\begin{align}
\label{Cl2.1}
\| \Lambda_+^{-1} D^{-\epsilon_1}  D_-^{\half-2\epsilon} (D_+^{\half+2\epsilon} \phi D_+ \psi) \|_{H^{s-1+\epsilon_1,-\half + 2\epsilon}} &\lesssim \| \Lambda_+ \phi\|_{H^{s-1,\half+\epsilon}} \| \Lambda_+ \psi\|_{H^{s-1,\half+\epsilon}} \, , \\
\label{Cl2.2}
\| \Lambda_+^{-1} D^{-\epsilon_1} ( D_+^{\half+2\epsilon} D_-^{\half-2\epsilon} \phi D_+ \psi) \|_{H^{s-1+\epsilon_1,-\half + 2\epsilon}} &\lesssim \| \Lambda_+ \phi\|_{H^{s-1,\half+\epsilon}} \| \Lambda_+ \psi\|_{H^{s-1,\half+\epsilon}} \, , \\
\label{Cl2.3}
\| \Lambda_+^{-1} D^{-\epsilon_1}  (D_+^{\half+2\epsilon} \phi D_-^{\half-2\epsilon} D_+ \psi) \|_{H^{s-1+\epsilon_1,-\half + 2\epsilon}} &\lesssim \| \Lambda_+ \phi\|_{H^{s-1,\half+\epsilon}} \| \Lambda_+ \psi\|_{H^{s-1,\half+\epsilon}} \, . 
\end{align}
{\bf Proof:} The singularity of $D^{-\epsilon_1}$ is harmless (\cite{T1}, Cor. 8.2) and may be replaced by $\Lambda^{-\epsilon_1}$ . \\
(\ref{Cl2.1}) reduces to
$$ \|\phi \psi \|_{H^{s-2,0}} \lesssim \|\phi\|_{H^{s-\half-2\epsilon,\half+\epsilon}} \|\psi\|_{H^{s-1,\half+\epsilon}} . $$
Use Prop. \ref{SML} and our assumption $s \ge \frac{3}{4} $ so that $(s-\half)+(s-1) \ge 0$ .\\ (\ref{Cl2.2}) and (\ref{Cl2.3}) can be handled similarly. \\[0.5em]
{\bf Claim 3:} In the case $n \ge 3$ the following estimate holds:
$$ \|D^{-1} \partial_l (A \phi \psi)\|_{H^{s-1,-\half+2\epsilon}} \lesssim \|DA\|_{H^{r-1,\half + \epsilon}} \|\phi\|_{H^{s,\half+\epsilon}} \|\psi\|_{H^{s,\half+\epsilon}} \, . $$
{\bf Proof:} We may replace $D$ by $\Lambda$ on the right hand side and use Prop. \ref{Prop.3.6} , which gives
\begin{align*}
\| A \phi \psi \|_{H^{s-1,-\half+2\epsilon}} & \lesssim \|A\|_{H^{r,\half+\epsilon}} \|\phi \psi\|_{H^{s-\half,0}} \\
& \lesssim \|A\|_{H^{r,\half+\epsilon}} \|\phi\|_{H^{s,\half+\epsilon}}  \psi\|_{H^{s,\half+\epsilon}} 
\end{align*}
with parameters $s_0=s-\half$ , $s_1=r$ , $s_2=1-s$ for the first estimate, so that $s_0+s_1+s_2 = r+\half > \frac{n-1}{2}$ and $s_0+s_1+s_2+s_1+s_2= 2r-s+\frac{3}{2} >\frac{n}{2}$ , because we assume $r > \frac{n}{2}-1$ and $2r-s > \frac{n-3}{2}$ . For the second estimate choose $s_0=\half-s$ , $s_1=s_2=s$ , so that $s_0+s_1+s_2 = s+\half > \frac{n-1}{2}$ and $s_1+s_2 =2s >\half$ . \\[0.5em]
{\bf Claim 4:} In the case $n=2$ the following estimate holds:
\begin{align*}
 &\|\Lambda_+^{-1} D^{1-\epsilon_1} (A \phi \psi)\|_{H^{s-1+\epsilon_1,-\half+2\epsilon}} \\
&\lesssim \|\Lambda_+ D^{1-\epsilon_1} A\|_{H^{r-2+\epsilon_1,\half + \epsilon}} \|\Lambda_+ \phi\|_{H^{s-1,\half+\epsilon}} \|\Lambda_+ \psi\|_{H^{s-1,\half+\epsilon}} \, . 
\end{align*} 
{\bf Proof:} Arguing as for claim 3 we may replace $D^{1-\epsilon_1}$ by $\Lambda^{1-\epsilon_1}$. We choose the same parameters and use $r+\half > \frac{3}{4}$ and $2r-s+\frac{3}{2} > 1$ by assumption for the first estimate and $s > \half$ for the second estimate.\\[0.5em]

Finally we have to consider the term $\partial_l (A_k |\phi|^2)$ . Using $\partial_t \phi = i \Lambda_m(\phi_+ - \phi_-)$ and $\partial_t A_k = iD(A_{k+} - A_{k-})$ we obtain
$$ \partial_t(A_k |\phi|^2) = i |\phi|^2 D(A_{k+}-A_{k-}) + i A_k \Lambda_m(\phi_+ - \phi_-) \overline{\phi} + A_k \phi \overline{i \Lambda_m(\phi_+ - \phi_-)} \, . $$
Now
$$ |\phi|^2 DA \precsim D(|\phi|^2 A) + A \phi D \overline{\phi} + A \overline{\phi} D \phi \, , $$
so that we only have to consider terms of the type $D(A_k |\phi|^2)$ and $ A_k \phi \Lambda \phi$ . The first term was considered in claim 3 and claim 4. Thus it remains to prove the following claim. \\
{\bf Claim 5:} In the case $n \ge 3$ the following estimate holds:
$$ \|D^{-1}(A\phi \Lambda \psi)\|_{H^{s-1,-\half+2\epsilon}} \lesssim \|DA\|_{H^{r-1,\half+\epsilon}} \|\phi\|_{H^{s,\half+\epsilon}}  \|\psi\|_{H^{s,\half+\epsilon}}$$
and in the case $n=2$ :
\begin{align*}
 &\|\Lambda_+^{-1} D^{-\epsilon_1}(A\phi \Lambda \psi)\|_{H^{s-1+\epsilon_1,-\half+2\epsilon}} \\
& \lesssim  \|\Lambda_+ D^{1-\epsilon_1} A\|_{H^{r-2+\epsilon_1,\half+\epsilon}} \|\Lambda_+ \phi\|_{H^{s-1,\half+\epsilon}}  \|\Lambda_+ \psi\|_{H^{s-1,\half+\epsilon}} \, .
\end{align*}
{\bf Proof:} By \cite{T1}, Cor. 8.2 we may replace powers of $D$ by powers of $\Lambda$ everywhere in the following argument, provided the frequencies of the product or $A$ are $\ge 1$ . Let us first consider the case $n \ge 3$ . By Prop. \ref{Prop.3.6} we obtain
\begin{align*}
\| A \phi \Lambda \psi\|_{H^{s-2,-\half + 2\epsilon}} & \lesssim \|A\|_{H^{r,\half+\epsilon}} \| \phi \Lambda \psi\|_{H^{s-\frac{5}{4},0}} \\
& \lesssim \|A\|_{H^{r,\half+\epsilon}} \| \phi \|_{H^{s,\half+\epsilon}} \| \psi \|_{H^{s,\half+\epsilon}}\, ,
\end{align*}
where we choose for the first estimate $s_0=s-\frac{5}{4}$ , $s_1=2-s$ , $s_2=r$ , so that $s_0+s_1+s_2 = r + \frac{3}{4} > \frac{n}{2} - \half$  using $r > \frac{n}{2}-\frac{5}{4}$ and $s_1+s_2 = r-s+2 \ge 1 $ using  $r \ge s-1 $.  For the second estimate we choose $s_0 = \frac{5}{4}-s$ , $s_1=s $ , $s_2=s-1$, thus we have $s_0+s_1+s_2 = s + \frac{1}{4} > \frac{n}{2}-\half$ and $s_1+s_2 = 2s-1 > \half$ using  $s > \frac{3}{4}$ . \\
Next consider the case $n=2$ .  By Theorem \ref{Theorem3} we obtain
\begin{align*}
\| A \phi \Lambda \psi\|_{H^{s-2,-\half + 2\epsilon}} & \lesssim \|A\|_{H^{r,\half+\epsilon}} \| \phi \Lambda \psi\|_{H^{s-1-,0}} \\
& \lesssim \|A\|_{H^{r,\half+\epsilon}} \| \phi \|_{H^{s,\half+\epsilon}} \| \psi \|_{H^{s,\half+\epsilon}}\, ,
\end{align*}
where we choose for the first estimate $s_0=s-1-$ , $s_1=2-s$ , $s_2=r$ , so that $s_0+s_1+s_2 = r + 1- > 1$ and $r+s-1 \ge 0$ for $s\ge \frac{3}{4}$ and $r > \frac{1}{4}$ .  For the second estimate we choose $s_0 = 1-s+$ , $s_1=s $ , $s_2=s-1$, thus we have $s_0+s_1+s_2 = s+  > \frac{3}{4}$ and $s_0 + 2(s_1+s_2) = 3s-1- > 1$ for  $s > \frac{2}{3}$ . \\
If the frequencies of the product as well as $A$ are $\le 1$ we estimate for $n \ge 3$ (and similarly for $n=2$)
\begin{align*}
\| D^{-1}(A\phi \Lambda \psi)\|_{H^{s-1,-\half + \epsilon}} & \lesssim \| D^{-1}(A\phi \Lambda \psi)\|_{H^{-3,0}} \lesssim \|A\phi\|_{H^{-2,0}} \|\Lambda \psi\|_{H^{s-1,\half+\epsilon}} \\
& \lesssim \|DA\|_{H^{r-1,\half+\epsilon}} \|\phi\|_{H^{s,\half+\epsilon}} \|\psi\|_{H^{s,\half+\epsilon}} \, , 
\end{align*}
where we used \cite{T1}, Cor 8.2 , again to replace $D^{-1}$ by $\Lambda^{-1}$ and $D$ by $ \Lambda$ , and Prop. \ref{SML} (or Prop. \ref{Prop.3.6}) with $s_0=4$ , $s_1=-2$ , $s_2=s-1$ for the second inequality , so that $s_0+s_1+s_2 >\frac{n}{2}$ , and $s_0=2$ , $s_1= r$ , $s_2=s$ , so that $s_0+s_1+s_2 = r+s+2 > n$ for the last inequality.

The proof of (\ref{4.12}) and (\ref{4.15}) is now complete. \\[0.5em]
It remains to prove (\ref{4.10}),(\ref{4.11}), (\ref{4.13}) and (\ref{4.14}). The properties (\ref{4.10}) and (\ref{4.13}) are given by (\ref{8}). 

Next we consider the case $n \ge 3$ and prove (\ref{4.11}). By (\ref{1}) we have
$$\partial_t F_{{0k}_{| t=0}} = - \partial_t F_{{k0}_{|t=0}} = - \partial^l F_{{kl}_{| t=0}} + j_{k_{| t=0}} \, . $$ 
By (\ref{8}) we have $\partial^l F_{{kl}_{| t=0}} \in H^{s-2}$. It remains to prove 
$$ j_{k_{| t=0}} = Im(\phi_0 \overline{\partial_k \phi_0}) + |\phi_0|^2 a_{0k} \in H^{s-2} \, . $$
First we obtain
$$ \|\phi_0 \overline{\partial_k \phi_0} \|_{H^{s-2}} \lesssim \|\phi_0\|_{H^s} \| \partial_k \phi_0\|_{H^{s-1}} < \infty $$
by Prop. \ref{SML}, because $s > \frac{n}{2} -1$ .

Concerning the term $|\phi_0|^2 a_{0k}$ we prove \\
{\bf Claim:}
 \begin{align*} \| |\phi_0|^2 a_{0k}\|_{H^{s-2}} & \lesssim \|\phi_0\|_{H^s}^2 \|Da_{0k}\|_{H^{r-1}} \quad \quad \quad \quad {\mbox for } \, n \ge 3 \\
 \| |\phi_0|^2 a_{0k}\|_{H^{s-2}} & \lesssim \|\phi_0\|_{H^s}^2 \|D^{1-\epsilon_1} a_{0k}\|_{H^{r-1+\epsilon_1}} \quad {\mbox for } \, n = 2 \, .
\end{align*}
{\bf Proof:} If $r \le \frac{n}{2}$ we obtain by Prop. \ref{SML} :
$$ \| |\phi_0|^2 a_{0k}\|_{H^{s-2}} \lesssim \| |\phi_0|^2\|_{H^{\frac{n}{2}-2+s-r+}} \|a_{0k}\|_{H^r} \lesssim \| \phi_0\|_{H^{s}}^2 \|a_{0k}\|_{H^r} \, , $$
where the first estimate directly follows from Prop. \ref{SML}, and the last estimate requires
$ s \ge \frac{n}{2}-2+s-r+ $ and $-\frac{n}{2}+2-s+r+2s > \frac{n}{2}$ , which hold by the assumptions $s,r >\frac{n}{2}-1$ . In the case $r > \frac{n}{2}$ we obtain by Prop. \ref{SML}
$$ \| |\phi_0|^2 a_{0k}\|_{H^{s-2}} \lesssim \| |\phi_0|^2\|_{H^{s-2}} \|a_{0k}\|_{H^r} \lesssim \| \phi_0\|_{H^{s}}^2 \|a_{0k}\|_{H^r} \, , $$
assuming $ s > \frac{n}{2}-1$ .

In the case of high frequencies $\ge 1$ of $a_{0k}$ this is enough for our claim to hold. Otherwise we have equivalent frequencies of $|\phi_0|^2 a_{0k}$ and $|\phi_0|^2$ , so that we obtain in the case $n \ge 3$ :
\begin{align*}
\| |\phi_0|^2 a_{0k}\|_{H^{s-2}} & \lesssim  \|\Lambda^{s-2}(|\phi_0|^2) a_{0k}\|_{L^2} \lesssim \|\Lambda^{s-2}(|\phi_0|^2)\|_{L^2} \|a_{0k}\|_{L^{\infty}} \\
& \lesssim \|\phi_0|^2\|_{H^{s-2}} \|D a_{0k}\|_{H^{\frac{n}{2}-1+}} \lesssim \|\phi_0\|_{H^s}^2 \|Da_{0k}\|_{H^{r-1}}
\end{align*}
and similarly for $n=2$ with $\|Da_{0k}\|_{H^{s-1}}$ replaced by $\|D^{1-\epsilon_1} a_{0k} \|_{H^{r-1+\epsilon_1}}$ .  \\
Moreover
$$\partial_t F_{jk} = \partial_0(\partial_j A_k - \partial_k A_j) = \partial_j(\partial_k A_0 + F_{0k}) - \partial_k(\partial_j A_0 + F_{0j}) = \partial_j F_{0k} - \partial_k F_{0j} \, ,$$
so that by (\ref{8}) we obtain
$$ \partial_t F_{{jk}_{| t=0}} = \partial_j F_{{0k}_{| t=0}} - \partial_k F_{{0j}_{| t=0}} \in H^{s-2} \, . $$
We are done in the case $n \ge 3$ .

Next we consider the case $n=2$ and prove (\ref{4.14}). By (\ref{1}) we have
$$D^{-\epsilon} \partial_t F_{{0k}_{| t=0}} = -D^{-\epsilon} \partial_t F_{{k0}_{| t=0}} = -D^{-\epsilon} \partial^l F_{{kl}_{| t=0}} +
D^{-\epsilon} j_{{k}_{| t=0}} \, . $$
By (\ref{8}) we have $D^{-\epsilon} \partial^l F_{{kl}_{| t=0}} \in H^{s-2+\epsilon}$, so that it remains to prove
$$D^{-\epsilon} j_{{k}_{| t=0}} = D^{-\epsilon} Im(\phi_0 \overline{\partial_k \phi_0}) + D^{-\epsilon} (|\phi_0|^2 a_{0k}) \in H^{s-2+\epsilon} \, . $$
First we show the estimate \\
{\bf Claim:} $$ \|D^{-\epsilon}(|\phi_0|^2 a_{0k})\|_{H^{s-2+\epsilon}} \lesssim \|\phi_0\|_{H^s}^2 \|D^{1-\epsilon} a_{0k}\|_{H^{r-1+\epsilon}} \, . $$
{\bf Proof:} For frequencies of the product $\ge 1$ this immediately follows from the proof of the previous claim. If the frequencies of the product is $\le 1$ and the frequencies of $a_{0k}$ are $\ge 1$ , we estimate assuming $ s > \half$ :
\begin{align*}
\| D^{-\epsilon}(|\phi_0|^2) a_{0k})\|_{H^{s-2+\epsilon}} & \lesssim \| |\phi_0|^2 a_{0k}\|_{L^1} \lesssim \| |\phi_0|^2\|_{L^2} \|a_{0k}\|_{L^2} \\
& \lesssim \| \phi_0\|_{H^s}^2 \|a_{0k}\|_{H^r} \lesssim \|\phi_0\|_{H^s}^2 \|D^{1-\epsilon} a_{0k}\|_{H^{r-1+\epsilon}} \, . 
\end{align*}
Finally, if both frequencies are $\le 1$ we obtain
\begin{align*}
\| D^{-\epsilon}(|\phi_0|^2) a_{0k})\|_{H^{s-2+\epsilon}} & \lesssim \| |\phi_0|^2 a_{0k}\|_{L^1} \lesssim \| |\phi_0|^2\|_{L^1} \|a_{0k}\|_{L^{\infty}} \\
& \lesssim \| \phi_0\|_{L^2}^2 \|D^{1-\epsilon}a_{0k}\|_{L^2} \lesssim \|\phi_0\|_{H^s}^2 \|D^{1-\epsilon} a_{0k}\|_{H^{r-1+\epsilon}} \, . 
\end{align*}

It remains to prove for $n=2$ and $s > \half$ :
$$ \|D^{-\epsilon}(\phi_0 \overline{\partial_k \phi_0})\|_{H^{s-2+\epsilon}} \lesssim \|\phi_0\|_{H^s} \|\partial_k \phi_0\|_{H^{s-1}} \, . $$
For large frequencies of the product this follows from Prop. \ref{SML}. If the product has frequencies $\le 1$ the frequencies of $\phi_0$ and $\overline{\partial_k \phi_0}$ are equivalent, so that
\begin{align*}
\|D^{-\epsilon}(\phi_0 \overline{\partial_k \phi_0})\|_{H^{s-2+\epsilon}} \lesssim \| \phi_0 \overline{\partial_k \phi_0}\|_{H^{-1,\frac{2}{1+\epsilon}}} &\lesssim \||\Lambda^{\half}\phi_0|^2 \|_{H^{-1,\frac{2}{1+\epsilon}}} \\
& \lesssim \||\Lambda^{\half}\phi_0|^2\|_{L^1} \lesssim \|\Lambda^{\half} \phi_0\|_{L^2}^2 \lesssim \|\phi_0\|_{H^s}^2 \, .
\end{align*}
Moreover, similarly as in the case $n \ge 3$ the assumption (\ref{8}) implies
$$ D^{-\epsilon} \partial_t F_{{jk}_{|t=0}} = D^{-\epsilon} \partial_j F_{{0k}_{|t=0}} - D^{-\epsilon} \partial_k F_{{0j}_{|t=0}} \in H^{s-2+\epsilon} \, . $$
Thus we have also shown (\ref{4.14}). The proof of Theorem \ref{Theorem1} is now complete.

\end{document}